\documentclass[leqno,10pt,a4paper]{amsart}
\usepackage{amsthm}
\usepackage[all]{xy}
\usepackage[margin=2.7cm]{geometry}
\usepackage{graphicx}
\usepackage{hyperref}
\usepackage{hyperref}
\hypersetup{
    colorlinks,
    linkcolor={red!50!black},
    citecolor={blue!50!black},
    urlcolor={blue!80!black}
}
\usepackage[usenames,dvipsnames]{color}
\usepackage{mleftright}
\usepackage{mathabx}
\usepackage{longtable}
\usepackage{booktabs}
\usepackage{colortbl}
\usepackage{amssymb}
\usepackage{tikz}
\newtheorem{thm}{Theorem}[section]
\newtheorem{lemma}[thm]{Lemma}
\newtheorem{cor}[thm]{Corollary}
\newtheorem{prop}[thm]{Proposition}
\theoremstyle{definition}
\newtheorem{example}[thm]{Example}
\newtheorem{remark}[thm]{Remark}

\newtheorem{definition}[thm]{Definition}

\newtheorem*{conjectureA}{Conjecture A}
\numberwithin{equation}{section}
\newcommand{\orig}{\mathbf{0}}
\newcommand{\Z}{{\mathbb{Z}}}
\newcommand{\Q}{{\mathbb{Q}}}
\newcommand{\C}{{\mathbb{C}}}
\newcommand{\R}{{\mathbb{R}}}
\newcommand{\FF}{{\mathbb{F}}}

\newcommand{\NQ}{N_\Q}
\newcommand{\MQ}{M_\Q}
\newcommand{\GL}{GL}
\newcommand{\SL}{SL}
\newcommand{\Proj}{\mathbb{P}}
\newcommand{\Hom}[1]{\operatorname{Hom}\mleft({#1}\mright)}
\newcommand{\Vol}[1]{\operatorname{Vol}\mleft({#1}\mright)}
\newcommand{\abs}[1]{\left\vert{#1}\right\vert}
\newcommand{\intr}[1]{#1^\circ}
\newcommand{\bdry}[1]{\partial #1}
\newcommand{\V}[1]{\mathcal{V}\mleft({#1}\mright)}
\newcommand{\F}[1]{\mathcal{F}\mleft({#1}\mright)}
\renewcommand{\third}{\frac{1}{3}(1,1)}
\newcommand{\conv}[1]{\operatorname{conv}\mleft({#1}\mright)}
\newcommand{\sconv}[1]{\operatorname{conv}\mleft\{{#1}\mright\}}
\newcommand{\cone}[1]{\operatorname{cone}\mleft({#1}\mright)}
\newcommand{\scone}[1]{\operatorname{cone}\mleft\{{#1}\mright\}}
\newcommand{\dual}[1]{{{#1}^*}}

\newcommand{\Newt}[1]{\operatorname{Newt}\mleft({#1}\mright)}
\newcommand{\Ehr}[1]{\operatorname{Ehr}_{#1}\mleft(t\mright)}
\newcommand{\denom}[1]{\operatorname{denom}\mleft({#1}\mright)}
\newcommand{\lcm}[1]{\operatorname{lcm}\mleft\{{#1}\mright\}}
\renewcommand{\gcd}[1]{\operatorname{gcd}\mleft\{{#1}\mright\}}
\renewcommand{\min}[1]{\operatorname{min}\mleft\{{#1}\mright\}}
\newcommand{\bmin}[1]{\operatorname{min}\mleft({#1}\mright)}
\renewcommand{\max}[1]{\operatorname{max}\mleft\{{#1}\mright\}}
\newcommand{\bmax}[1]{\operatorname{max}\mleft({#1}\mright)}

\newcommand{\mP}{{m_P}}
\newcommand{\cB}{\mathcal{B}}
\newcommand{\cA}{\mathcal{A}}
\newcommand{\mB}{{m_\cB}}
\newcommand{\dB}{{d_\cB}}
\newcommand{\sB}{{s_\cB}}
\newcommand{\rE}{r_E}
\newcommand{\SC}[1]{\operatorname{SC}\mleft({#1}\mright)}
\newcommand{\res}[1]{\operatorname{res}\mleft({#1}\mright)}

\newcommand{\modb}[1]{\left(\operatorname{mod}\ {#1}\right)}
\newcommand{\mult}[1]{\operatorname{mult}\mleft({#1}\mright)}
\newcommand{\mut}{\operatorname{mut}}
\newcommand{\hmin}{h_{\operatorname{min}}}
\newcommand{\hmax}{h_{\operatorname{max}}}
\newcommand{\umin}{u_{\operatorname{min}}}
\newcommand{\const}[1]{\mathrm{coeff}_1({#1})}
\newcommand{\Hilb}[1]{\mathrm{Hilb}\mleft({#1}\mright)}
\newcommand{\Sing}[1]{\mathrm{Sing}\mleft({#1}\mright)}
\newcommand{\Tlattice}[1]{\Gamma_{#1}}
\newcommand{\codim}[1]{\mathrm{codim}\mleft({#1}\mright)}
\renewcommand{\dim}[1]{\mathrm{dim}\mleft({#1}\mright)}
\renewcommand{\emptyset}{\varnothing}
\renewcommand{\phi}{\varphi}
\newcommand{\proofsection}[1]{%
	\vspace{0.2em}%
	\noindent\underline{{#1}:}%
 \ }
\newcolumntype{g}{>{\columncolor[gray]{0.95}\centering\arraybackslash}m{1.5em}}
\newcolumntype{w}{>{\centering\arraybackslash}m{1.5em}}
\newcommand{\gap}{\hspace{0.5em}}
\newcommand{\vgap}{\vspace{0.1em}}
\newcommand{\padding}{\rule[-1.45ex]{0pt}{0.2em}\gap}
\newcommand{\oddrow}{\rowcolor[gray]{0.95}}
\newcommand{\evnrow}{}
\newcommand{\xmapsto}[1]{\stackrel{#1}{\longmapsto}}
\newcommand{\createRpoly}[2]{%
	\expandafter\newcommand\csname Rpoly#1\endcsname{R_{#2}}%
}
\createRpoly{P}{1}
\createRpoly{L}{2}
\createRpoly{O}{3}
\createRpoly{M}{4}
\createRpoly{N}{5}
\createRpoly{E}{6}
\createRpoly{F}{7}
\createRpoly{G}{8}
\createRpoly{B}{9}
\createRpoly{H}{10}
\createRpoly{C}{11}
\createRpoly{A}{12}
\createRpoly{I}{13}
\createRpoly{D}{14}
\createRpoly{J}{15}
\createRpoly{K}{16}
\newcommand{\createTpoly}[2]{%
	\expandafter\newcommand\csname Tpoly#1\endcsname{T_{#2}}%
}
\createTpoly{H}{1}
\createTpoly{E}{2}
\createTpoly{A}{3}
\createTpoly{G}{4}
\createTpoly{D}{5}
\createTpoly{I}{6}
\createTpoly{B}{7}
\createTpoly{C}{8}
\createTpoly{F}{9}
\newcommand{\createPpoly}[2]{%
	\expandafter\newcommand\csname Ppoly#1\endcsname{P_{#2}}%
}
\createPpoly{J}{1}
\createPpoly{G}{2}
\createPpoly{H}{3}
\createPpoly{I}{4}
\createPpoly{B}{5}
\createPpoly{E}{6}
\createPpoly{F}{7}
\createPpoly{D}{8}
\createPpoly{A}{9}
\createPpoly{C}{10}
\graphicspath{{images/}}
\begin{document}
%-------------------------------------------------------------------------------
\author[Kasprzyk]{Alexander Kasprzyk}
\author[Nill]{Benjamin Nill}
\author[Prince]{Thomas Prince}
\address{Department of Mathematics\\Imperial College London\\180 Queen's Gate\\London, SW7 2AZ, UK}
\email{a.m.kasprzyk@imperial.ac.uk} % Kasprzyk
\email{t.prince12@imperial.ac.uk} % Prince
\address{Department of Mathematics\\Stockholm University\\SE-$106$\ $91$\ Stockholm\\Sweden}
\email{nill@math.su.se} % Nill
%-------------------------------------------------------------------------------
%-------------------------------------------------------------------------------
\title{Minimality and mutation-equivalence of polygons}
\maketitle
%-------------------------------------------------------------------------------
\begin{abstract}
We introduce a concept of minimality for Fano polygons. We show that, up to mutation, there are only finitely many Fano polygons with given singularity content, and give an algorithm to determine the mutation-equivalence classes of such polygons. This is a key step in a program to classify orbifold del~Pezzo surfaces using mirror symmetry. As an application, we classify all Fano polygons such that the corresponding toric surface is qG-deformation-equivalent to either (i)~a smooth surface; or (ii)~a surface with only singularities of type $\third$.
\end{abstract}
%-------------------------------------------------------------------------------
\section{Introduction}\label{sec:intro}
%-------------------------------------------------------------------------------
\subsection{An introduction from the viewpoint of algebraic geometry and mirror symmetry}\label{sec:del_pezzo_intro}
%-------------------------------------------------------------------------------
A \emph{Fano polygon} $P$ is a convex polytope in $\NQ:=N\otimes_\Z\Q$, where $N$ is a rank-two lattice, with primitive vertices $\V{P}$ in $N$ such that the origin is contained in its strict interior, $\orig\in\intr{P}$. A Fano polygon defines a toric surface $X_P$ given by the \emph{spanning fan} of $P$; that is, $X_P$ is defined by the fan whose cones are spanned by the faces of $P$. The toric surface $X_P$ has cyclic quotient singularities (corresponding to the cones over the edges of $P$) and the anti-canonical divisor $-K_X$ is $\Q$-Cartier and ample. Hence $X_P$ is a toric \emph{del~Pezzo surface}.

The simplest example of a toric del~Pezzo surface is $\Proj^2$, corresponding, up to $\GL_2(\Z)$-equivalence, to the triangle $P=\sconv{(1,0),(0,1),(-1,-1)}$. It is well-known that there are exactly five smooth toric del~Pezzo surfaces, and that these are a subset of the sixteen toric Gorenstein del~Pezzo surfaces (in bijective correspondence with the famous sixteen reflexive polygons~\cite{Bat85,Rab89}). More generally, if one bounds the \emph{Gorenstein index} $r$ (the smallest positive integer such that $-rK_X$ is very ample) the number of possibilities is finite. Dais classified those toric del~Pezzo surfaces with Picard rank one and $r\leq 3$~\cite{Dai09}. A general classification algorithm was presented in~\cite{KKN08}.

A new viewpoint on del~Pezzo classification is suggested by mirror symmetry. We shall sketch this briefly; for details see \cite{ProcECM}. An $n$-dimensional Fano variety $X$ is expected to correspond, under mirror symmetry, to a Laurent polynomial $f\in\C[x_1^{\pm1},\ldots,x_n^{\pm1}]$~\cite{Auroux,Ba04,QC105,ProcECM}. Under this correspondence, the regularised quantum period $\widehat{G}_X$ of $X$ -- a generating function for Gromov--Witten invariants -- coincides with the classical period $\pi_f$ of $f$ -- a solution of the associated Picard--Fuchs differential equation -- given by
$$\pi_f(t)=\left(\frac{1}{2\pi i}\right)^n\int_{\abs{x_1}=\ldots=\abs{x_n}=1}\frac{1}{1-tf(x_1,\ldots,x_n)} \frac{dx_1}{x_1} \cdots \frac{dx_n}{x_n}=\sum_{k\geq 0}\const{f^k} t^k.$$
If a Fano variety $X$ is mirror to a Laurent polynomial $f$ then it is expected that $X$ admits a degeneration to the singular toric variety $X_P$ associated to the Newton polytope $P$ of $f$.

In general there will be many (often infinitely many) different Laurent polynomials mirror dual to $X$, and hence many toric degenerations $X_P$. It is conjectured that these Laurent polynomials are related via birational transformations analogous to cluster transformations, which are called \emph{mutations}~\cite{ACGK12,FZ,GU10,GHK13}. A mutation acts on the Newton polytope $P:=\Newt{f}\subset\NQ$ of a Laurent polynomial via ``rearrangement of Minkowski slices''~(see~\S\ref{subsec:mutation_N}), and on the dual polytope $\dual{P}\subset\MQ$, $M:=\Hom{N,\Z}$, via a piecewise-$\GL_n(\Z)$ transformation~(see~\S\ref{subsec:mutation_M})~\cite{ACGK12}. At the level of Laurent polynomials, if $f$ and $g$ are related via mutation then their classical periods agree~\cite[Lemma~2.8]{ACGK12}: $\pi_f=\pi_g$. Ilten~\cite{Ilt12} has shown that if two Fano polytopes $P$ and $Q$ are related by mutation then the corresponding toric varieties $X_P$ and $X_Q$ are deformation equivalent: there exists a flat family $\mathcal{X}\rightarrow\Proj^1$ such that $\mathcal{X}_0\cong X_P$ and $\mathcal{X}_\infty\cong X_Q$. In fact $X_P$ and $X_Q$ are related via a \emph{$\Q$-Gorenstein (qG) deformation}~\cite{Pragmatic}.

Classifying Fano polygons up to mutation-equivalence thus becomes a fundamental problem. One important mutation invariant is the \emph{singularity content}~(see~\S\ref{subsec:sing_content})~\cite{AK14}.  This consists of a pair $(n,\cB)$, where $n$ is an integer -- the number of \emph{primitive $T$-singularities} -- and $\cB$ is a \emph{basket} -- a collection of so-called \emph{residual singularities}. A residual singularity is a cyclic quotient singularity that is rigid under qG-deformations; at the other extreme a $T$-singularity is a cyclic quotient singularity that admits a qG-smoothing~\cite{KS-B88}. The toric del~Pezzo surface $X_P$ is qG-deformation-equivalent to a  del~Pezzo surface $X$ with singular points given by $\cB$ and Euler number of the non-singular locus $X\setminus\Sing{X}$ equal to $n$.

\begin{definition}[\cite{Pragmatic}]\label{defn:TG}
A del~Pezzo surface with cyclic quotient singularities that admits a qG-degeneration (with reduced fibres) to a normal toric del~Pezzo surface is said to be of class \emph{TG}.
\end{definition}

Notice that not all del~Pezzo surfaces can be of class TG: not every del~Pezzo has $h^0(X,-K_X)>0$, for example (see Example~\ref{eg:degree_hilb_third_one_one}). But it is natural to conjecture the following:

\begin{conjectureA}[\cite{Pragmatic}]\hypertarget{conj:A}{}
There exists a bijective correspondence between the set of mutation-equivalence classes of Fano polygons and the set of qG-deformation-equivalence classes of locally qG-rigid class TG del~Pezzo surfaces with cyclic quotient singularities.
\end{conjectureA}

The main results of this paper can be seen as strong evidence in support of the conjecture above. First, an immediate consequence of Theorem~\ref{thm:T-sing_mutation_classes} is:

\begin{thm}
There are precisely ten mutation-equivalence classes of Fano polygons such that the toric del~Pezzo surface $X_P$ has only $T$-singularities. They are in bijective correspondence with the ten families of smooth del~Pezzo surfaces.
\end{thm}

\noindent
Second, combining the results of~\cite{CH} with Theorem~\ref{thm:classes_third_one_one} we have:

\begin{thm}
There are precisely $26$ mutation-equivalence classes of Fano polygons with singularity content $(n,\{m\times\third\})$, $m\geq 1$. They are in bijective correspondence with the $26$ qG-deformation families of del~Pezzo surfaces with $m\times\third$ singular points that admit a toric degeneration.
\end{thm}

\noindent
In Theorem~\ref{thm:general_minimals} we prove that the number of Fano polygons with basket $\cB$ is finite, and give an algorithm for their classification.  If one accepts Conjecture~\hyperlink{conj:A}{A} then this tells us that, for fixed basket $\cB$, the number of qG-deformation-equivalence classes of del~Pezzo surfaces of type TG with singular points $\cB$ is finite, and gives an algorithm for classifying their toric degenerations.

%-------------------------------------------------------------------------------
\subsection{An introduction from the viewpoint of cluster algebras and cluster varieties}\label{sec:cluster_algebra_intro}
%-------------------------------------------------------------------------------
One can obtain information about mutations of polygons using quivers and the theory of cluster algebras~\cite{FZ,FG09}. There is a precise analogy between mutation classes of Fano polygons and the clusters of certain cluster algebras, as we now describe. Let $L\cong\Z^n$, and fix a skew-symmetric form $\{\cdot,\cdot\}$ on $L$. A \emph{cluster} $C$ is a transcendence basis for $\C(L)$, and a \emph{seed} is a pair $(B,C)$ where $B$ is a basis of $L$. There is a notion of \emph{mutation} of seeds, given in Definition~\ref{def:seed_mutation} below; this depends on the form $\{\cdot,\cdot\}$. A \emph{cluster algebra} is the algebra generated by all clusters that can be obtained from a given initial seed by mutation. To a seed $(B,C)$ one can associate a quiver $Q_B$ with vertex set equal to $B$ and the number of arrows from $e_i \in B$ to $e_j \in B$ equal to $\bmax{\{e_i,e_j\},0}$. Changing the seed $(B,C)$ by a mutation changes the quiver $Q_B$ by a \emph{quiver mutation} (see Definition~\ref{def:quiver_mutation}). Conversely, from a quiver with vertex set $B$ and no vertex-loops or two-cycles, one can construct a cluster algebra by setting $L=\Z^B$, defining $\{e_i,e_j\}$ to be the (signed) number of arrows from $e_i \in B$ to $e_j\in B$, and taking the initial seed to be $(B,C)$, where $C$ is the standard transcendence basis for $\C(L)$.

We can also associate a quiver and a cluster algebra to a Fano polygon $P$ as follows. Suppose that the singularity content of $P$ is $(n,\cB)$. The associated quiver $Q_P$ has $n$ vertices; each vertex $v$ corresponds to a \emph{primitive $T$-singularity}, and hence determines an edge $E$ of $P$ (these edges need not be distinct). The number of arrows from vertex $v$ to vertex $v'$ is defined to be $\max{w \wedge w',0}$, where $\wedge$ denotes the determinant, and $w$ and $w'$ are the primitive inner normal vectors to the edges $E$ and $E'$ of $P$. The cluster algebra $\cA_P$ associated to $P$ is the cluster algebra associated to $Q_P$; we denote the initial seed of this cluster algebra by $(B_P,C_P)$.

We show in Proposition~\ref{prop:seed_mutations} below that a mutation from a seed $(B,C)$ to a seed $(B',C')$ induces a mutation between the corresponding Fano polygons $P$ and $P'$.  We then show, in  Proposition~\ref{prop:quiver_mutations}, that a mutation from a Fano polygon $P$ to a Fano polygon $P'$ induces a mutation between the corresponding quivers $Q_P$ and $Q_{P'}$.  These correspondences have consequences for mutation-equivalence which are not readily apparent from the polygon alone. 

\begin{example} \label{example:pentagon}
Consider a Fano polygon $P \subset \NQ$ containing only two primitive $T$-singularities, and suppose that the corresponding inner normal vectors form a basis for the dual lattice~$M$. Then there are at most five polygons, up to $\GL_2(\Z)$-equivalence, that are mutation-equivalent to $P$. This follows from the facts that the quiver associated to $P$ has underlying graph $A_2$, and that the exchange graph of the $A_2$ cluster algebra is pentagonal; see Corollary~\ref{cor:pentagon} below.
\end{example}

%-------------------------------------------------------------------------------
\subsection{An introduction from the viewpoint of the geometry of numbers}\label{sec:geom_numbers_intro}
%-------------------------------------------------------------------------------
The relation between the lattice points in a convex body and its geometric shape and volume is a key problem in convex geometry and integer optimisation. These connections have been addressed specifically for lattice polytopes, independently of their significance in toric geometry. Here we focus only on the case of interest in this paper, that of a Fano polygon. A classical result in this area is the following (these statements can be generalised and quantified, see~\cite{HS09}):

\begin{thm}[\cite{LZ91,Sco76}]
There are only finitely many $\GL_2(\Z)$-equivalence classes of Fano polygons $P$ with $I$ interior lattice points, for each $I\in\Z_{>0}$.
\end{thm}

\begin{cor}
There are only a finite number of possibilities for the area and number of lattice points of a Fano polygon with $I$ interior lattice points, for each $I\in\Z_{>0}$.
\end{cor}

In~\cite{ACGK12} a new equivalence relation on Fano polytopes was introduced, called \emph{mutation-equivalence}, that is weaker than $\GL_2(\Z)$-equivalence. In particular there exist infinitely many mutation-equivalent Fano polytopes that are not $\GL_2(\Z)$-equivalent (see, for example,~\cite[Example~3.14]{AK13}) and so their area and number of lattice points cannot be bounded. Mutation-equivalence does, however, preserve the Ehrhart series (and hence volume) of the dual polytope $\dual{P}\subset\MQ$ (see~\S\ref{subsec:mutation_M})~\cite[Proposition~4]{ACGK12}.

\begin{remark}\label{rem:arbitrary_large_volume}
The reader should be aware that there is no direct relation between the volume of $P$ and that of $\dual{P}$. Whilst the product $\Vol{P}\cdot\Vol{\dual{P}}$ of the (normalised) volume of $P$ and of its dual polygon $\dual{P}$ cannot be arbitrarily small \cite{Mah39}, both can simultaneously become arbitrarily large. For example, $P_k:=\sconv{(k,1),(k,-1),(-1,0)}$ for $k \in \Z_{>0}$.
\end{remark}

In the two-dimensional case, given a Fano polygon $P$ there exists an explicit formula for the Ehrhart series and volume of the dual polygon $\dual{P}$ given in terms of the \emph{singularity content} of $P$ (see~\S\ref{subsec:sing_content})~\cite{AK14}. A consequence of this formula is that mutation-equivalent Fano polygons cannot have an arbitrarily large number of vertices~\cite[Lemma~3.8]{AK14}.

Recall that the \emph{height} $\rE\in\Z_{\geq 0}$ of a lattice line segment $E\subset\NQ$ is the lattice distance of $E$ from the origin, and the \emph{width} is given by the positive integer $k=\abs{E\cap N}-1$. Clearly there exist unique non-negative integers $n$ and $k_0$, $0\leq k_0<\rE$, such that $k=n\rE+k_0$. Suppose that $E$ is the edge of a Fano polygon, so that the vertices of $E$ are primitive. As described in~\cite{AK14} (see also~\S\ref{subsec:sing_content}), one can decompose $E$ into $n+1$ (or $n$ if $k_0=0$) lattice line segments with primitive vertices. Of these, $n$ line segments have their width equal to their height; the cones over these line segments correspond to \emph{primitive $T$-singularities}. If $k_0\ne 0$ then there is one additional lattice line segment of width $k_0<\rE$; the cone over this line segment corresponds to a \emph{residual singularity}.

Although there may be several different decompositions of this form for an edge $E$, it turns out that the residual singularity is unique -- it does not depend on the choice of decomposition. In addition, the collection of residual singularities arising from all of the edges of $P$, which we call the \emph{basket} of $P$ and denote by $\cB$, is a mutation invariant. We say that a lattice point of $P$ is \emph{residual} if it lies in the strict interior of a residual cone, for some fixed choice of decomposition. The number of residual lattice points does not depend on the chosen decomposition, but only on the basket $\cB$ of $P$, and is invariant under mutation. The main results of this paper can now be stated in a way analogous to the classical results above:

\begin{thm}
There are only finitely many mutation-equivalence classes of Fano polygons $P$ with $N$ residual lattice points, for each $N\in\Z_{\ge 0}$.
\end{thm}

\begin{proof}
If there are no residual cones then the result follows from Theorem~\ref{thm:T-sing_minimals}. In order to use Theorem~\ref{thm:general_minimals} we only have to show that the height $\rE$ of an edge $E$ containing a residual cone is bounded. Let $v_1$ and $v_2$ be primitive points on $E$ such that $\scone{v_1,v_2}$ is a residual cone. The line segment joining $v_1$ and $v_2$ has width $1\le k<\rE$. We see that the lattice triangle
$$\sconv{\orig,v_1,v_1+(v_2-v_1)/k}$$
has at most $N+3$ lattice points. Pick's formula~\cite{Pick} implies that its area and thus the height $\rE$ is bounded in terms of $N$.
\end{proof}

\begin{cor}
There are only a finite number of possibilities for the dual area and the number of vertices of a Fano polygon with $N$ residual lattice points, for each $N\in\Z_{\ge 0}$.
\end{cor}

\noindent
In particular, this shows that there exist no Fano polygons with empty basket but an arbitrarily large number of vertices. Note that it can be easily seen that there exist centrally-symmetric Fano polygons with an arbitrarily large number of vertices where every edge corresponds to a residual singularity.

%-------------------------------------------------------------------------------
\section{Mutation of Fano polygons}\label{sec:mutation}
%-------------------------------------------------------------------------------
In~\cite[\S3]{ACGK12} the concept of mutation for a lattice polytope was introduced. We state it here in the simplified case of a Fano polygon $P\subset\NQ$ and refer to~\cite{ACGK12} for the general definitions.

%-------------------------------------------------------------------------------
\subsection{Mutation in $N$}\label{subsec:mutation_N}
%-------------------------------------------------------------------------------
Let $w\in M:=\Hom{N,\Z}$ be a primitive inner normal vector for an edge $E$ of $P$, so $w:N\rightarrow\Z$ induces a grading on $\NQ$ and $w(v)=-\rE$ for all $v\in E$, where $\rE$ is the height of $E$. Define
$$\hmax:=\max{w(v)\mid v\in P}\quad\text{ and }\quad\hmin:=-\rE=\min{w(v)\mid v\in P}.$$
We have that $\hmax>0$ and $\hmin<0$. For each $h\in\Z$ we define $w_h(P)$ to be the (possibly empty) convex hull of those lattice points in $P$ at height $h$,
$$w_h(P):=\sconv{v\in P\cap N\mid w(v)=h}.$$
By definition $w_{\hmin}(P)=E$ and $w_{\hmax}(P)$ is either a vertex or an edge of $P$. Let $v_E\in N$ be a primitive lattice element of $N$ such that $w(v_E)=0$, and define $F:=\sconv{\orig,v_E}$, a line segment of unit width parallel to $E$ at height $0$. Notice that $v_E$, and hence $F$, is uniquely defined only up to sign.

\begin{definition}\label{defn:mutation}
Suppose that for each negative height $\hmin\leq h<0$ there exists a (possibly empty) lattice polytope $G_h\subset\NQ$ satisfying
\begin{equation}\label{eq:slice}
\{v\in\V{P}\mid w(v)=h\}\subseteq G_h+\abs{h}F\subseteq w_h(P).
\end{equation}
where `$+$' denotes the Minkowski sum, and we define $\emptyset+Q=\emptyset$ for any polygon $Q$.  We call $F$ a \emph{factor} of $P$ with respect to $w$, and define the \emph{mutation} given by the primitive normal vector $w$, factor $F$, and polytopes $\{G_h\}$ to be:
$$\mut_w(P,F):=\conv{\bigcup_{h=\hmin}^{-1}G_h\cup\bigcup_{h=0}^{\hmax}\left(w_h(P)+hF\right)}\subset\NQ.$$
\end{definition}

Although not immediately obvious from the definition, the resulting mutation is independent of the choices of $\{G_h\}$~\cite[Proposition~1]{ACGK12}. Furthermore, up to isomorphism, mutation does not depend on the choice of $v_E$: we have that $\mut_w(P,F)\cong\mut_w(P,-F)$. Since we consider a polygon to be defined only up to $\GL_2(\Z)$-equivalence, mutation is well-defined and unique. Any mutation can be inverted by inverting the sign of $w$: if $Q:=\mut_w(P,F)$ then $P=\mut_{-w}(Q,F)$~\cite[Lemma~2]{ACGK12}. Finally, we note that $P$ is a Fano polygon if and only if the mutation $Q$ is a Fano polygon~\cite[Proposition~2]{ACGK12}.

We call two polygons $P$ and $Q\subset\NQ$ \emph{mutation-equivalent} if there exists a finite sequence of mutations between the two polygons (considered up to $\GL_2(\Z)$-equivalence). That is, if there exist polygons $P_0, P_1,\ldots, P_n$ with $P\cong P_0$, $P_{i+1}=\mut_{w_i}(P_i,F_i)$, and $Q\cong P_n$, for some $n\in\Z_{\geq 0}$.

\begin{remark}\label{rem:two_dim_special}
We remark briefly upon the three ways in which our definition above differs slightly from that in~\cite{ACGK12}.
\begin{enumerate}
\item\label{item:two_dim_special_origin}
First,~\cite{ACGK12} does not require that the factor $F$ be based at the origin. The condition that $\orig\in\V{F}$ is harmless, and indeed we have touched on this above when we noted that $F$ and $-F$ give $\GL_2(\Z)$-equivalent mutations: in general translation of a factor $F$ by a lattice point $v\in w^\perp$, where $w^\perp:=\{v\in\NQ\mid w(v)=0\}$, results in isomorphic mutations. It is reasonable to regard a factor as being defined only up to translation by elements in $w^\perp\cap N$, with the resulting mutation defined only up to ``shear transformations'' fixing the points in $w^\perp$.
\item\label{item:two_dim_special_dim_1}
Second, the more general definition places no restriction on the dimension of the factor $F$, although the requirement that $F\subset w^\perp$ does mean that $\codim{F}\geq 1$. In particular it is possible to take $F=v$, where $w(v)=0$. But observe that $v=\orig+v$, and $\mut_w(\,\cdot\,,\orig)$ is the identity, hence $\mut_w(P,v)$ is trivial. Thus our insistence that $\dim{F}=1$ is reasonable.
\item\label{item:two_dim_special_primitive}
Finally, our requirement that $F$ is of unit width is a natural simplification: in general, if the factor can be written as a Minkowski sum $F=F_1+F_2$, where we can insist that each $F_i\subset w^\perp$ and $\dim{F_i}>0$, then the mutation with factor $F$ can be written as the composition of two mutations with factors $F_1$ and $F_2$ (with fixed $w$). Thus it is reasonable to assume that the factor is Minkowski-indecomposable and hence, for us, a primitive line segment.
\end{enumerate}
\end{remark}

In two dimensions, mutations are completely determined by the edges of $P$:

\begin{lemma}\label{lem:mutation_iff_length}
Let $E$ be an edge of $P$ with primitive inner normal vector $w\in M$. Then $P$ admits a mutation with respect to $w$ if and only if $\abs{E\cap N}-1\geq\rE$.
\end{lemma}
\begin{proof}
Let $k:=\abs{E\cap N}-1$ be the width of $E$. At height $h=\hmin=-\rE$, condition~\eqref{eq:slice} becomes $E=G_{\hmin}+\rE F$. Hence this condition can be satisfied if and only if $k\geq\rE$. Suppose that $k\geq\rE$ and consider the cone $C:=\cone{E}$ generated by $E$. At height $\hmin<h<0$, $h\in\Z$, the line segment $C_h:=\{v\in C\mid w(v)=h\}\subset\NQ$ (with rational end-points) has width $\abs{h}k/\rE\geq\abs{h}$. Hence $w_h(C)\subset w_h(P)$ has width at least $\abs{h}-1$. Suppose that there exists some $v\in\V{P}$ such that $w(v)=h$. Since $v\not\in w_h(C)$ we conclude that $w_h(P)$ has width at least $\abs{h}$. Hence condition~\eqref{eq:slice} can be satisfied. If $\{v\in\V{P}\mid w(v)=h\}=\emptyset$ then we can simply take $G_h=\emptyset$ to satisfy condition~\eqref{eq:slice}.
\end{proof}

%-------------------------------------------------------------------------------
\subsection{Mutation in $M$}\label{subsec:mutation_M}
%-------------------------------------------------------------------------------
Given a Fano polygon $P\subset\NQ$ we define the \emph{dual polygon}
$$\dual{P}:=\{u\in\MQ\mid u(v)\geq -1\text{ for all }v\in P\}\subset\MQ.$$
In general this has rational-valued vertices and necessarily contains the origin in its strict interior. Define $\varphi:\MQ\rightarrow\MQ$ by $u\mapsto u-\umin w$, where $\umin:=\min{u(v)\mid v\in F}$. Since $F=\sconv{\orig,v_E}$, this is equivalent to
$$
\varphi(u)=
\begin{cases}
u,&\text{ if }u(v_E)\geq 0;\\
u - u(v_E)w,&\text{ if }u(v_E)< 0.
\end{cases}
$$
This is a piecewise-$\GL_2(\Z)$ map, partitioning $\MQ$ into two half-spaces whose common boundary is generated by $w$. Crucially~\cite[Proposition~4]{ACGK12}:
$$\varphi(\dual{P})=\dual{Q},\qquad\text{ where }Q:=\mut_w(P,F).$$
An immediate consequence of this is that the volume and Ehrhart series of the dual polygons are preserved under mutation: $\Vol{\dual{P}}=\Vol{\dual{Q}}$ and $\Ehr{\dual{P}}=\Ehr{\dual{Q}}$. Equivalently, mutation preserves the anti-canonical degree and Hilbert series of the corresponding toric varieties: $(-K_{X_P})^2=(-K_{X_Q})^2$ and $\Hilb{X_P,-K_{X_P}}=\Hilb{X_Q,-K_{X_Q}}$.

\begin{example}\label{eg:mutations_of_P2}
Consider the polygon $P_{(1,1,1)}:=\sconv{(1,1),(0,1),(-1,-2)}\subset\NQ$. The toric variety corresponding to $P_{(1,1,1)}$ is $\Proj^2$. Let $w=(0,-1)\in M$, so that $\hmin=-1$ and $\hmax=2$, and set $F=\sconv{\orig,(1,0)}\subset\NQ$. Then $F$ is a factor of $P_{(1,1,1)}$ with respect to $w$, giving the mutation $P_{(1,1,2)}:=\mut_w(P_{(1,1,1)},F)$ with vertices $(0,1)$, $(-1,-2)$, $(1,-2)$ as depicted below. The toric variety corresponding to $P_{(1,1,2)}$ is $\Proj(1,1,4)$.

\begin{center}
\begin{tabular}{r@{ }c@{ }c@{ }c@{ }r}
\raisebox{25px}{$\NQ:$}&
\includegraphics[scale=0.6]{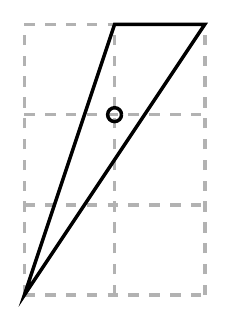}& %original height=92px
\raisebox{25px}{$\longmapsto$}&
\includegraphics[scale=0.6]{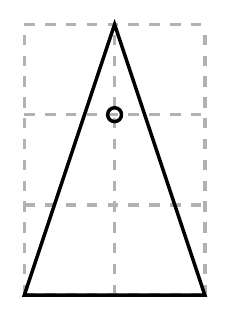}& %original height=92px
\phantom{$\NQ:$}
\end{tabular}
\end{center}

\noindent
In $\MQ$ we see the mutation as a piecewise-$\GL_2(\Z)$ transformation. This acts on the left-hand half-space $\{(u_1,u_2)\in\MQ\mid u_1<0\}$ via the transformation
$$(u_1,u_2)\mapsto(u_1,u_2)\small\begin{pmatrix}
1&-1\\
0&1
\end{pmatrix}$$
and on the right-hand half-space via the identity.

\begin{center}
\begin{tabular}{r@{ }c@{ }c@{ }c@{ }r}
\raisebox{25px}{$\MQ:$}&
\includegraphics[scale=0.6]{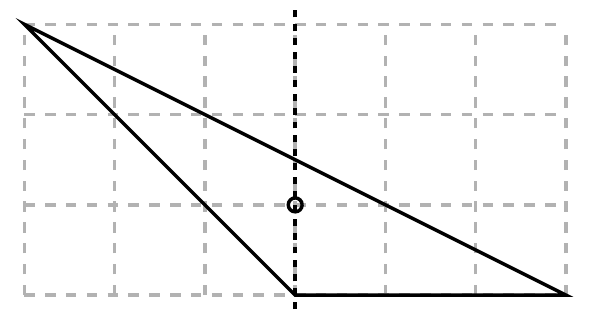}& %original height=92px
\raisebox{25px}{$\longmapsto$}&
\includegraphics[scale=0.6]{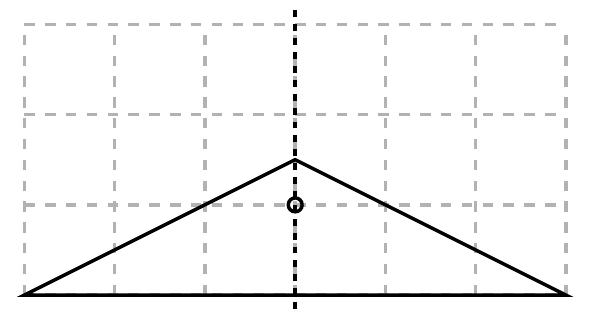}& %original height=92px
\phantom{$\MQ:$}
\end{tabular}
\end{center}

\noindent
We can draw a graph of all possible mutations obtainable from $P_{(1,1,1)}$: the vertices of the graph denote $\GL_2(\Z)$-equivalence classes of Fano polygons, and two vertices are connected by an edge if there exists a mutation between the two Fano polygons (notice that, since mutations are invertible, we can regard the edges as being undirected). We obtain a tree whose typical vertex is trivalent~\cite[Example~3.14]{AK13}:
\begin{center}
\small
\begin{tikzpicture}[grow=down,level distance=1cm]
\tikzstyle{level 3}=[sibling distance=6cm]
\tikzstyle{level 4}=[sibling distance=3cm]
\tikzstyle{level 5}=[sibling distance=1.5cm,level distance=0.7cm]
\node {$(1,1,1)$}
child {node {$(1,1,2)$}
child {node {$(1,2,5)$}
child {node {$(2,5,29)$}
child {node {$(5,29,433)$}
child {node {} edge from parent[dotted]}
child {node {} edge from parent[dotted]}}
child {node {$(2,29,169)$}
child {node {} edge from parent[dotted]}
child {node {} edge from parent[dotted]}}}
child {node {$(1,5,13)$}
child {node {$(5,13,194)$}
child {node {} edge from parent[dotted]}
child {node {} edge from parent[dotted]}}
child {node {$(1,13,34)$}
child {node {} edge from parent[dotted]}
child {node {} edge from parent[dotted]}}}}
};
\end{tikzpicture}
\end{center}
\noindent
Here the vertices have been labelled with weights $(a,b,c)$, where the polygon $P_{(a,b,c)}$ corresponds to the toric variety $\Proj(a^2,b^2,c^2)$. The triples $(a,b,c)$ are solutions to the Markov equation
$$3xyz=x^2+y^2+z^2,$$
and each mutation corresponds, up to permutation of $a$, $b$, and $c$, to a transformation of the form
$(a,b,c)\mapsto (3bc-a,b,c)$. In the theory of Markov equations these transformations are also called \emph{mutations}. A solution $(a,b,c)$ is called \emph{minimal} if $a+b+c$ is minimal, and every solution can be reached via mutation from a minimal solution. Minimal solutions correspond to those triangles with $\Vol{P_{(a,b,c)}}$ minimal. In this example $(1,1,1)$ is the unique minimal solution. These statements can be generalised to any mutation between triangles~\cite{AK13}. Hacking--Prokhorov~\cite{HP10} use these minimal solutions in their classification of rank-one qG-smoothable del~Pezzo surfaces of class TG.
\end{example}

%-------------------------------------------------------------------------------
\section{Invariants of Fano polygons}\label{sec:invariants}
%-------------------------------------------------------------------------------
We wish to be able to establish whether or not two Fano polygons are mutation-equivalent. In this section we introduce two mutation invariants of a Fano polygon $P\subset\NQ$: singularity content, discussed in~\S\ref{subsec:sing_content} below, can be thought of as studying the part of $P$ that remains untouched by mutation (the basket $\cB$ of residual singularities); the cluster algebra $\cA_P$, discussed in~\S\ref{subsec:quivers_and_clusters} below, studies the part of $P$ that changes under mutation (the primitive $T$-singularities). Although we have no proof, it seems likely that together these two invariants completely characterise the mutation-equivalence classes. Finally, in~\S\ref{subsec:affine_manifolds} we briefly mention the connection with affine manifolds and the Gross--Seibert program~\cite{GS-PROGRAM}.

%-------------------------------------------------------------------------------
\subsection{Singularity content}\label{subsec:sing_content}
%-------------------------------------------------------------------------------
In~\cite{AK14} the concept of \emph{singularity content} for a Fano polygon was introduced. First we state the definition for a cyclic quotient singularity $\frac{1}{R}(a,b)$, where $\gcd{R,a}=\gcd{R,b}=1$. (Recall that $\frac{1}{R}(a,b)$ denotes the germ of a quotient singularity $\C^2/\mu_R$, where $\varepsilon\in\mu_R$ acts via $(x,y)\mapsto(\varepsilon^ax,\varepsilon^by)$.) Let $k,r,c\in\Z$ be non-negative integers such that $k=\gcd{R,a+b}$, $R=kr$, and $a+b=kc$. Then $r$ is equal to the \emph{Gorenstein index} of the singularity, and $k$ is called the \emph{width}. Thus $\frac{1}{R}(a,b)$ can be written in the form $\frac{1}{kr}(1,kc-1)$ for some $c\in\Z$ with $\gcd{r,c}=1$.

\begin{definition}[$k\divides r$]\label{defn:T_singularity}
A cyclic quotient singularity such that $k=nr$ for some $n\in\Z_{>0}$, that is, a cyclic quotient singularity of the form $\frac{1}{nr^2}(1,nrc-1)$, is called a \emph{$T$-singularity} or a singularity of \emph{class~$T$}. When $n=1$, so that the singularity is of the form $\frac{1}{r^2}(1,rc - 1)$, we call it a \emph{primitive $T$-singularity}.
\end{definition}

\begin{definition}[$k<r$]\label{defn:R_singularity}
A cyclic quotient singularity of the form $\frac{1}{kr}(1,kc-1)$ with $k<r$ is called a \emph{residual singularity} or a singularity of \emph{class~$R$}.
\end{definition}

$T$-singularities appear in the work of Wahl~\cite{Wah80} and Koll\'ar--Shepherd-Barron~\cite{KS-B88}. A singularity is of class $T$ if and only if it admits a qG-smoothing. At the opposite extreme, a singularity is of class $R$ if and only if it is rigid under qG-deformation. More generally, consider the cyclic quotient singularity $\sigma=\frac{1}{kr}(1,kc-1)$. Let $0\leq k_0<r$ and $n$ be the unique non-negative integers such that $k=nr+k_0$. Then either $k_0=0$ and $\sigma$ is qG-smoothable, or $k_0>0$ and $\sigma$ admits a qG-deformation to the residual singularity $\frac{1}{k_0r}(1,k_0c-1)$~\cite{AK14,Pragmatic}. This motivates the following definition:

\begin{definition}[\protect{\cite[Definition~2.4]{AK14}}]\label{defn:residue_sing_content_cone}
With notation as above, let $\sigma=\frac{1}{kr}(1,kc-1)$ be a cyclic quotient singularity. The \emph{residue} of $\sigma$ is given by
$$
\res{\sigma}:=
\begin{cases}
\emptyset&\text{ if }k_0=0,\\
\frac{1}{k_0r}(1,k_0c-1)&\text{ otherwise}.
\end{cases}
$$
The \emph{singularity content} of $\sigma$ is given by the pair $\SC{\sigma}:=\left(n,\res{\sigma}\right)$.
\end{definition}

\begin{example}
Let $\sigma=\frac{1}{nr^2}(1,nrc-1)$ be a $T$-singularity. Then $\SC{\sigma}=(n,\emptyset)$.
\end{example}

Singularity content has a natural description in terms of the cone $C$ defining the singularity. We call a two-dimensional cone $C\subset\NQ$ a \emph{$T$-cone} (respectively \emph{primitive $T$-cone}) if the corresponding cyclic quotient singularity is a $T$-singularity (respectively primitive $T$-singularity), and we call $C$ an \emph{$R$-cone} if the corresponding singularity is a residual singularity. Let $C=\scone{\rho_0,\rho_1}\subset\NQ$ be a two-dimensional cone with rays generated by the primitive lattice points $\rho_0$ and $\rho_1$ in $N$. The line segment $E=\sconv{\rho_0,\rho_1}$ is at height $r$ and has width $\abs{E\cap N}-1=k$. Write $k=nr+k_0$. Then there exists a partial crepant subdivision of $C$ into $n$ cones $C_1,\ldots,C_n$ of width $r$  and, if $k_0\neq 0$, one cone $C_0$ of width $k_0<r$. Although not immediately obvious from this description, the singularity corresponding to the $R$-cone $C_0$ is well-defined and equal to $\frac{1}{k_0r}(1,k_0c-1)$. The singularities corresponding to the $n$ primitive $T$-cones $C_1,\ldots,C_n$ depend upon the particular choice of subdivision: see~\cite[Proposition~2.3]{AK14} for the precise statement.

\begin{definition}\label{defn:sing_content_polygon}
Let $P\subset\NQ$ be a Fano polygon with edges $E_1,\ldots,E_m$, numbered cyclically, and let $\sigma_1,\ldots,\sigma_m$ be the corresponding two-dimensional cyclic quotient singularities $\sigma_i=\cone{E_i}$, with $\SC{\sigma_i}=\left(n_i,\res{\sigma_i}\right)$. The \emph{singularity content} of $P$, denoted by $\SC{P}$, is the pair $(n,\cB)$, where $n:=\sum_{i=1}^mn_i$ and $\cB$ is the cyclically-ordered list $\left\{\res{\sigma_i}\mid 1\leq i\leq m, \res{\sigma_i}\neq\emptyset\right\}$. We call $\cB$ the \emph{basket} of residual singularities of $P$.
\end{definition}

\noindent
Singularity content is a mutation invariant of $P$~\cite[Proposition~3.6]{AK14}. Intuitively one can see this from Lemma~\ref{lem:mutation_iff_length}: mutation removes a line segment of length $\abs{\hmin}$ from the edge at height $\hmin$, changing the corresponding singularity content by $(n,\res{\sigma})\mapsto(n-1,\res{\sigma})$; mutation adds a line segment of length $\hmax$ at height $\hmax$, changing the singularity content by $(n',\res{\sigma'})\mapsto(n'+1,\res{\sigma'})$. Put another way, mutation removes a primitive $T$-cone at height $\hmin$ and adds a primitive $T$-cone at height $\hmax$, leaving the residual cones unchanged. We can rephrase Lemma~\ref{lem:mutation_iff_length} in terms of singularity content:

\begin{lemma}\label{lem:mutation_iff_T-cone}
Let $E$ be an edge of $P$ with primitive inner normal vector $w\in M$, and let $\left(n,\res{\sigma}\right)$ be the singularity content of $\sigma=\cone{E}$. Then $P$ admits a mutation with respect to $w$ if and only if $n\ne 0$.
\end{lemma}

Singularity content provides an upper-bound on the maximum number of vertices of any polygon $P$ with $\SC{P}=(n,\cB)$, or, equivalently, an upper-bound on the Picard rank $\rho$ of the corresponding toric variety $X_P$~\cite[Lemma~3.8]{AK14}:
\begin{align}\label{eq:bound_on_rank}
\abs{\V{P}}\leq n+\abs{\cB}, &&\rho\leq n+\abs{\cB}-2.
\end{align}

\begin{example}\label{eg:P113}
Consider $P_{(1,1)}=\sconv{(1,0),(0,1),(-1,-3)}$ with corresponding toric variety $\Proj(1,1,3)$. This has singularity content $(2,\{\third\})$, hence, by~\eqref{eq:bound_on_rank}, any polygon $Q$ mutation-equivalent to $P_{(1,1)}$ has three vertices. Mutations between triangles were characterised in~\cite{AK13}: the mutation graph is given by
\begin{center}
\small
\begin{tikzpicture}[grow=right,level distance=2cm]
\node {$(1,1)$}
child {node {$(1,4)$}
child {node {$(4,19)$}
child {node {$(19,91)$}
child {node {$(91,436)$}
child {node {$\cdots$}}}}}
};
\end{tikzpicture}
\end{center}
Here the vertices have been labelled by pairs $(a,b)\in\Z_{>0}^2$, and correspond to $\Proj(a^2,b^2,3)$ and its associated triangle. These pairs are solutions to the Diophantine equation $5xy=x^2+y^2+3$. Up to exchanging $a$ and $b$, a mutation of triangles corresponds to the mutation $(a,b)\mapsto (5b-a,b)$ of solutions. There is a unique minimal solution given by $(1,1)$.
\end{example}

As noted in~\S\ref{sec:del_pezzo_intro}, the toric variety $X_P$ is qG-deformation-equivalent to a del~Pezzo surface $X$ with singular points $\cB$ and topological Euler number of $X\setminus\Sing{X}$ equal to $n$~\cite{AK14,Pragmatic}. The degree and Hilbert series can be expressed purely in terms of singularity content. Recall that information about a minimal resolution of a singularity $\sigma=\frac{1}{R}(1,a-1)$ is encoded in the Hirzebruch--Jung continued fraction expansion $[b_1,\ldots,b_s]$ of $R/(a-1)$; see, for example,~\cite{Ful93}. For each $i\in\{1,\ldots,s\}$ we inductively define the positive integers $\alpha_i,\beta_i$ as follows:
\begin{align*}
\alpha_1=\beta_s&=1,\\
\alpha_i/\alpha_{i-1}&:=[b_{i-1},\ldots,b_1],\quad 2\leq i\leq s,\\
\beta_i/\beta_{i+1}&:=[b_{i+1},\ldots,b_s],\quad 1\leq i\leq s-1.
\end{align*}
The values $-b_i$ give the self-intersection numbers of the exceptional divisors of the minimal resolution of $\sigma$, and the values $d_i:=-1+(\alpha_i+\beta_i)/R$ give the discrepancies. The degree contribution of $\sigma$ is given by:
$$A_\sigma:=s+1-\sum_{i=1}^sd_i^2b_i+2\sum_{i=1}^{s-1}d_id_{i+1}.$$
The Riemann--Roch contribution $Q_\sigma$ of $\sigma$ can be computed in terms of Dedekind sums (see~\cite[\S8]{Reid85}):
$$
Q_\sigma=\frac{1}{1-t^R}\sum_{i=1}^{R-1}\left(\delta_{ai}-\delta_0\right)t^{i-1},
\qquad\text{ where }\qquad
\delta_j:=\frac{1}{R}\sum_{\stackrel{\scriptstyle \varepsilon\in\mu_R}{\!\varepsilon\neq 1}}\frac{\varepsilon^j}{(1-\varepsilon)(1-\varepsilon^{a-1})}.
$$
\begin{prop}[\protect{\cite[Proposition~3.3 and Corollary~3.5]{AK14}}]\label{prop:sing_content_degree}
Let $P\subset\NQ$ be a Fano polygon with singularity content $(n,\cB)$. Let $X_P$ be the toric variety given by the spanning fan of $P$. Then
$$
(-K_{X_P})^2=12-n-\sum_{\sigma\in\cB}A_\sigma
\quad\text{ and }\quad
\Hilb{X_P,-K_{X_P}}=\frac{1+\left((-K_{X_P})^2-2\right)t+t^2}{(1-t)^3}+\sum_{\sigma\in\cB}Q_\sigma.
$$
\end{prop}

The terms $Q_\sigma$ can be interpreted as a periodic correction to the initial term
$$\frac{1+\left((-K_{X_P})^2-2\right)t+t^2}{(1-t)^3}=
\sum_{i\geq 0}\left({i+1\choose 2}(-K_{X_P})^2+1\right)t^i.$$
Set $Q_{\text{num}}:=(1-t^R)Q_\sigma$. The contribution from $Q_\sigma$ at degree $i$ is equal to the coefficient of $t^m$ in $Q_{\text{num}}$, where $i\equiv m\modb{R}$.

\begin{example}\label{eg:degree_hilb_third_one_one}
Let $P\subset\NQ$ be a Fano polygon with singularity content $\left(n,\left\{m\times\third\right\}\right)$, for some $n\in\Z_{\geq 0}$, $m\in\Z_{>0}$. We see that $A_{\third}=5/3$ and $Q_{\third}=-t / 3(1-t^3)$, giving
\begin{align*}
\Vol{\dual{P}}=(-K_{X_P})^2&=12-n-\frac{5m}{3}\qquad\text{ and }\\
\Ehr{\dual{P}}=\Hilb{X_P,-K_{X_P}}&=\frac{1+(11-n-2m)t+(12-n-m)t^2+(11-n-2m)t^3+t^4}{(1-t^3)(1-t)^2}.
\end{align*}
In particular, for any $i\ge 0$,
$$
\abs{i\dual{P}\cap M}=h^0\left(X_P,-iK_{X_P}\right)=
{i+1\choose 2}(-K_{X_P})^2+1-
\begin{cases}
\frac{m}{3}&\text{ if }i\equiv 1\modb{3},\\
0&\text{ otherwise.}
\end{cases}
$$
Since $\abs{\dual{P}\cap M}=13-n-2m\geq 1$ we have that $0\leq n\leq 10$ and $1\leq m\leq6-n/2$. Notice that $\left(5,\{4\times\third\}\right)$, $\left(3,\{5\times\third\}\right)$, and $\left(1,\{6\times\third\}\right)$ give $h^0=0$, hence there cannot exist a corresponding Fano polygon $P$ (or toric surface $X_P$). They do, however, correspond to the Euler numbers and singular points of the del~Pezzo surfaces $X_{4,1/3}$, $X_{5,2/3}$, and $X_{6,7/2}$, respectively, in~\cite{CH}. These three del~Pezzo surfaces cannot be of class TG.
\end{example}

%-------------------------------------------------------------------------------
\subsection{The sublattice $\Tlattice{P}$ of $M$}\label{subsec:T_sublattice}
%-------------------------------------------------------------------------------
There exist examples of Fano polygons that have the same singularity content, but that are not mutation-equivalent (see Example~\ref{eg:smooth_case} below); there is additional structure in the arrangement of the primitive $T$-singularities that singularity content ignores.

\begin{definition}\label{defn:normal_vector_sublattice}
Let $P\subset\NQ$ be a Fano polygon. Define $\Tlattice{P}$ to be the sublattice of $M$ generated by the primitive inner normal vectors to the edges of $P$, so that
$$\Tlattice{P}:=\left<\widebar{u}\mid u\in\V{\dual{P}}\right>$$
where $\widebar{u}\in M$ is the unique primitive lattice vector generating the ray passing through $u$. Let $[M:\Tlattice{P}]$ denote the index of this sublattice in $M$.
\end{definition}

\begin{lemma}\label{lem:sublattice_index_preserved}
Let $P\subset\NQ$ be a Fano polygon and let $Q:=\mut_w(P,F)$ be a mutation of $P$. Then $[M:\Tlattice{P}] = [M:\Tlattice{Q}]$.
\end{lemma}

\begin{proof}
Recall from \S\ref{subsec:mutation_M} that mutation acts on the element of $M$ via the piecewise-linear map $\varphi:u\mapsto u-\umin w$, where $\umin:=\min{u(v_F)\mid v_F\in\V{F}}$. Since $w\in\Tlattice{P}$ (by definition $w$ is a primitive inner normal to an edge of $P$), $\varphi$ maps $\Tlattice{P}$ to itself. Mutations are invertible, with the inverse to $\mut_w(\,\cdot\,,F)$ being $\mut_{-w}(\,\cdot\,,F)$. Thus the index is preserved.
\end{proof}

\begin{example}\label{eg:smooth_case}
Consider the reflexive polygons
\begin{align*}
\RpolyJ:=&\sconv{(1,0),(0,1),(-1,-1),(0,-1)}\subset\NQ\\
\text{and }\RpolyK:=&\sconv{(1,0),(0,1),(-1,0),(0,-1)}\subset\NQ.
\end{align*}
These are depicted in Figure~\ref{fig:polygon_smoothable} on page~\pageref{fig:polygon_smoothable}. The toric varieties defined via the spanning fan are the del Pezzo surfaces $\FF_1$ and $\Proj^1\times\Proj^1$ respectively. Both \hyperlink{fig:RpolyJ}{$\RpolyJ$} and \hyperlink{fig:RpolyK}{$\RpolyK$} have singularity content $(4,\emptyset)$. The primitive normal vectors to the edges of $\RpolyJ$ generate all of $M$, whereas the normal vectors of the edges of $\RpolyK$ generate an index two sublattice $\Tlattice{P}=\left<(1,1),(-1,1)\right>\subset M$. Lemma~\ref{lem:sublattice_index_preserved} shows that $\RpolyJ$ and $\RpolyK$ are not mutation-equivalent. We will give another proof of this, using quiver mutations, in Example~\ref{eg:smooth_case_2} below.
\end{example}

%-------------------------------------------------------------------------------
\subsection{Quivers and cluster algebras}\label{subsec:quivers_and_clusters}
%-------------------------------------------------------------------------------
We first recall the definition of cluster algebra~\cite{FG09,FZ}, having fixed a rank-$n$ lattice $L$ and a skew-symmetric form $\{\cdot,\cdot\}$.

\begin{definition}\label{def:seed}
A \emph{seed} is a pair $(B,C)$ where $B$ is a basis of $L$ and $C$ is a transcendence basis of $\C(L)$, referred to as a \emph{cluster}.
\end{definition}

\begin{definition}\label{def:seed_mutation}
 Given a seed $(B,C)$ with $B= \{e_1, \ldots , e_n\}$ and $C = \{x_1,\ldots,x_n\}$, the \emph{$j$th mutation} of $(B,C)$ is the seed $(B',C')$, where $B' = \{e_1',\ldots,e_n'\}$ and $C' = \{x'_1,\ldots,x'_n\}$ are defined by:
$$
e'_k = \begin{cases}
-e_j, & \text{if $k = j$} \\
e_k + \bmax{b_{kj},0}e_j, & \text{otherwise}
\end{cases}
$$
 where $b_{kl} = \{e_k,e_l\}$,
\begin{align*}
x_k' = x_k\text{ if }k \ne j,&&
\text{ and }&&
x_jx'_j = \!\!\prod_{\stackrel{\scriptstyle k\text{ such that}}{b_{jk} > 0}}\!\!{x^{b_{jk}}_k} + \!\!\prod_{\stackrel{\scriptstyle l\text{ such that}}{b_{jl} < 0}}\!\!{x^{b_{lj}}_l}.
\end{align*}
\end{definition}

Recall from~\S\ref{sec:cluster_algebra_intro} that there is a quiver $Q_P$ and a cluster algebra $\cA_P$ associated to a Fano polygon $P$.  Let $\SC{P}=(n,\cB)$ and fix a numbering of the $n$ primitive $T$-cones in the spanning fan of $P$.  The $i$th primitive $T$-cone in the spanning fan of $P$ corresponds to a vertex $v_i$ of $Q_P$, and thus corresponds to a basis element $e_i$ in $B_P$, where $(B_P, C_P)$ is the initial seed for $\cA_P$.  Let $E_i$ denote the edge of $P$ determined by the $i$th primitive $T$-cone; note that different primitive $T$-cones can determine the same edge.

\begin{prop}[Mutations of seeds induce mutations of polygons] \label{prop:seed_mutations}
Let $P$ be a Fano polygon with singularity content $(n,\cB)$.  Denote by $w_i \in M$ the primitive inner normal vector to $E_i$.  Consider the map $\pi\colon L\rightarrow M$ such that $\pi(e_i) = w_i$ for each~$i$.  Let $(B_P,C_P)$ be the inital seed for $\cA_P$, and write $B = \{e_1,\ldots,e_n\}$, $C = \{x_1,\ldots,x_n\}$.  Let $(B',C')$ be the $j$th mutation of $(B_P,C_P)$, and write $B' = \{e_1',\ldots,e_n'\}$.  Then $\{\pi(e'_1),\ldots,\pi(e'_n)\}$ are the primitive inner normal vectors to the edges of $P'$, where $P'$ is the mutation of $P$ determined by $w_j$. Furthermore, every mutation $P'$ of $P$ arises in this way.
\end{prop}
\begin{proof}
This is a straightforward calculation using mutation in $M$ (see~\S\ref{subsec:mutation_M}).
\end{proof}

There is a well-known notion of quiver mutation, going back to Bernstein--Gelfand--Ponomarev~\cite{BGP73}, Fomin--Zelevinsky~\cite{FZ}, and others.

\begin{definition}\label{def:quiver_mutation}
Given a quiver $Q$ and a vertex $v$ of $Q$, the \emph{mutation of $Q$ at $v$} is the quiver $\mut(Q,v)$ obtained from $Q$ by:
\begin{enumerate}
\item\label{item:quiver_mutation_add_shortcuts}
adding, for each subquiver $v_1 \to v \to v_2$, an arrow from $v_1$ to $v_2$;
\item\label{item:quiver_mutation_delete_2_cycles}
deleting a maximal set of disjoint two-cycles;
\item\label{item:quiver_mutation_reverse_arrows}
reversing all arrows incident to $v$.
\end{enumerate}
The resulting quiver is well-defined up to isomorphism, regardless of the choice of two-cycles in~\eqref{item:quiver_mutation_delete_2_cycles}.
\end{definition}

\begin{prop}[Mutations of polygons induce mutations of quivers]\label{prop:quiver_mutations}
Let $P$ be a Fano polygon, let $v$ be a vertex of $Q_P$ corresponding to a primitive $T$-cone in $P$, and let $P'$ be the corresponding mutation of $P$. We have $Q_{P'} = \mut(Q_P,v)$.
\end{prop}

\begin{proof}
Let $E$ denote the edge of $P$ determined by the primitive $T$-cone corresponding to $v$, and let $w\in M$ denote the primitive inner normal vector to $E$. Mutation with respect to $w$ acts on $M$ as a piecewise-linear transformation that is the identity in one half-space, and on the other half-space is a shear transformation $u\mapsto u+(w\wedge u)w.$ Thus determinants between the pairs of normal vectors change as follows:
\begin{enumerate}
\item
The inner normal vector $w$ to the mutating edge $E$ becomes $-w$, so that all arrows into $v$ change direction;
\item
For a pair of normal vectors in the same half-space (as defined by $w$), the determinant does not change;
\item Consider primitive $T$-cones with inner normal vectors in different half-spaces (as defined by $w$), let the corresponding vertices of $Q_P$ be $v_1$ and $v_2$, and let the corresponding inner normal vectors in $M$ be $w_1$ and $w_2$.  Without loss of generality we may assume that $w_1\wedge w>0$ and $w_2\wedge w < 0$, so that there are arrows $v_1\rightarrow v\rightarrow v_2$ in $Q_P$.  Under mutation, the primitive inner normal vectors change as $w_1\mapsto w_1'$, $w_2\mapsto w_2'$ where $w_1'=w_1$, $w_2'=w_2 + (w\wedge w_2)w$.  Thus:
$$w_1' \wedge w_2'=w_1\wedge w_2+(w\wedge w_2)(w_1\wedge w)$$
and so we add an arrow for each path $v_1\rightarrow v\rightarrow v_2$. Cancelling two-cycles results in precisely the result of calculating the signed total number of arrows from $v_1$ to $v_2$.
\end{enumerate}
Observing finally that if $v_1$,~$v_2$ give normal vectors in the same half-space then there are no paths $v_1\rightarrow v\rightarrow v_2$ or $v_2\rightarrow v\rightarrow v_1$, we see that this description coincides with that of a quiver mutation.
\end{proof}

Propositions~\ref{prop:seed_mutations} and~\ref{prop:quiver_mutations} give upper and lower bounds on the mutation graphs of Fano polygons. For example:

\begin{cor}[see Example~\ref{example:pentagon}] \label{cor:pentagon}
If a Fano polygon $P$ has singularity content $(2,\cB)$ and is such that the primitive inner normal vectors of the two edges corresponding to the two primitive $T$-cones form a basis of the dual lattice $M$, then the mutation-equivalence class of $P$ has at most five members.
\end{cor}
\begin{proof}
The quiver associated to $P$ is simply the $A_2$ quiver. The cluster algebra $\cA_P$ is therefore well-known (being the cluster algebra associated to the $A_2$ quiver) and its cluster exchange graph forms a pentagon. Note however that the \emph{quiver} mutation graph is trivial, as the $A_2$ quiver mutates only to itself. Proposition~\ref{prop:seed_mutations} implies that the mutation graph of $P$ has at most five vertices. (Proposition~\ref{prop:quiver_mutations} does not give a non-trivial lower bound here: indeed polygon~\hyperlink{poly:7}{7} in Figure~\ref{fig:third_one_one_polygons} on page~\pageref{fig:third_one_one_polygons} gives an example of such a polygon $P$ with trivial mutation graph.)
\end{proof}

\begin{example}\label{eg:smooth_case_2}
For the polygons \hyperlink{fig:RpolyJ}{$\RpolyJ$} and \hyperlink{fig:RpolyK}{$\RpolyK$} considered in Example~\ref{eg:smooth_case} above the associated quivers $Q_{\RpolyJ}$ and $Q_{\RpolyK}$ are:
$$
\begin{array}{c@{\hspace{1.5cm}}c}
\xymatrix{
\bullet \ar@2{->}[r] & \bullet \ar@3{->}[d] \\
\bullet \ar[u] \ar[ur] & \bullet \ar@2{->}[l] \ar[ul]
}&
\xymatrix{
\bullet \ar@2{->}[r] & \bullet \ar@2{->}[d]\\
\bullet \ar@2{->}[u] & \bullet \ar@2{->}[l]
}\\
\phantom{{}_{\RpolyJ}}Q_{\RpolyJ}&\phantom{{}_{\RpolyK}}Q_{\RpolyK}
\end{array}
$$
Observe that for $Q_{\RpolyK}$ the number of arrows between any two vertices is even.  It is easy to see that this property is preserved under mutation. Therefore the quivers $Q_{\RpolyJ}$ and $Q_{\RpolyK}$ are not mutation-equivalent, and so the polygons $\RpolyJ$ and $\RpolyK$ are not mutation-equivalent.
\end{example}

%-------------------------------------------------------------------------------
\subsection{Affine manifolds}\label{subsec:affine_manifolds}
%-------------------------------------------------------------------------------
This section sketches a more geometric approach to finding mutation invariants of polygons.  A detailed description of this material will be presented in forthcoming work of Prince~\cite{PrincePhD}, based on the foundational papers of Gross--Siebert~\cite{GS-INV} and Kontsevich--Soibelman~\cite{KS06}. In broad terms, one wishes to generalise the fact that the dual polygon $\dual{P}$ is the base of a special Lagrangian torus fibration given by the moment map, by allowing more general bases and more general torus fibrations. Specifically, the base of a special Lagrangian torus fibration carries an \emph{affine structure}~\cite{KS06,GS-PROGRAM}.  Deforming the torus fibration then leads to a deformation of the ``candidate base'' (actually exhibiting  a fibration over the deformed base is fraught with difficulty); the deformed base does not have the structure of a polygon, but rather is an \emph{affine manifold}.

\begin{definition}
A two-dimensional \emph{integral affine manifold} $B$ is a manifold which admits a maximal atlas with transition functions in $GL_2(\Z)\rtimes \Z^2$
\end{definition}

\begin{example}
The interior of $\dual{P}$ is an example of an integral affine manifold, covered as it is by a single chart. The polygon $\dual{P}$ itself is an example of an \emph{affine manifold with corners}.
\end{example}

Allowing singular fibres in the special Lagrangian torus fibration corresponds roughly to the base manifold acquiring \emph{focus-focus singularities}. We say ``roughly'' because, following~\cite{GS-PROGRAM}, we should be considering toric \emph{degenerations} rather than torus fibrations; see~\cite{GS-INV,KS06}. The local model for such a singularity is $\R^2\setminus\{\orig\}$, regarded as an affine manifold via a cover with two charts
$$
U_1=\R^2\setminus\R_{\leq 0}\times\{0\}\qquad
\text{ and }\qquad
U_2=\R^2\setminus\R_{\geq 0}\times\{0\}
$$
and the transition function
$$
(x,y)\mapsto
\begin{cases}
(x,y),&\text{ if }y>0;\\
(x + y,y),&\text{ if }y<0.
\end{cases}
$$
Sliding singularities, which Kontsevich--Soibelman call \emph{moving worms}~\cite{KS06}, gives an affine analogue of the deformations of the varieties.  In particular, allowing singularities to collide with boundary points of the affine manifold creates corners and provides an analogue for the toric qG-degenerations of the surface. This process, and its lifting via the Gross--Siebert program to construct the corresponding degeneration of algebraic varieties, will be described in~\cite{PrincePhD}.

Given the dual polygon $\dual{P}$, the set of singularities that we can introduce is in natural bijection with the primitive $T$-cones appearing in the singularity content of $P$. Consider the following process. Take a $T$-cone $\sigma$ of $P$, and introduce the corresponding singularity into the interior of $\dual{P}$, partially smoothing that corner and forming an affine manifold $B$. Now slide this singularity along the monodromy-invariant line, all the way to the opposite of $B$, and so forming a polygon $\dual{P'}$ with dual polygon $P'$.  

\begin{lemma}
$P'$ is equal to the mutation of $P$ defined by the chosen $T$-cone $\sigma$.
\end{lemma}

\begin{proof}
Mutation induces a piecewise-linear transformation on the dual polygon. In fact this corresponds exactly to the transition function between the two charts defining $B$. As the singularity approaches a corner, one of the charts covers all but a line segment with vanishing length, hence the polygons are related by the piecewise-linear transition function applied to the entire polygon.
\end{proof}

Given a polygon $\dual{P}$ one can introduce a maximal set of singularities and thus form an affine manifold $B$. Regarding two affine manifolds which differ by moving singularities along monodromy-invariant lines as equivalent, we see that every polygon in the same mutation class is equivalent as an affine manifold.  This gives a mutation invariant -- the affine manifold $B$ -- which we can use to distinguish minimal polygons. Specifically, we can compare the respective \emph{monodromy representations}~\cite{DBRANES09}. Since the transition functions of $B$ are in $GL_2(\Z)$ we can define parallel transport of integral vector fields. Fixing a basepoint in $B$, parallel transport around loops gives a representation $\pi_1(B)\rightarrow\SL_2(\Z)$. Properties of this representation may be used to distinguish different mutation classes of polygons. For example, in the cases considered in Example~\ref{eg:smooth_case}, one monodromy representation is surjective and the other is not.

%-------------------------------------------------------------------------------
\section{Minimal Fano polygons}\label{sec:minimality}
%-------------------------------------------------------------------------------
Given a polygon $P\subset\NQ$ we want to find a preferred representative in the mutation-equivalence class of $P$. Let $\bdry{P}$ denote the boundary of $P$ and let $\intr{P}:=P\setminus\bdry{P}$ denote the strict interior of $P$. We introduce the following definition:

\begin{definition}\label{defn:minimaility}
We call a Fano polygon $P\subset\NQ$ \emph{minimal} if for every mutation $Q:=\mut_w(P,F)$ we have that $\abs{\bdry{P}\cap N}\leq\abs{\bdry{Q}\cap N}$.
\end{definition}

Minimality is a local property of the mutation graph. It is certainly possible for there to exist more than one minimal polygon in a given mutation-equivalence class (see Example~\ref{eg:reflexive_are_minimal} below), however the number is finite. This is shown in Theorem~\ref{thm:T-sing_minimals} in the case when $\cB=\emptyset$, and in Theorem~\ref{thm:general_minimals} when $\cB\neq\emptyset$. Given a polygon $P\subset\NQ$ one can easily construct a mutation-equivalent minimal polygon. Set $P_0:=P$ and recursively define $P_{i+1}$ as follows. Let $\Gamma_i:=\{P_i\}\cup\{\mut_w(P_i,F)\mid\text{ for all possible }w\in M\}$ where, as usual, we regard a polygon as being defined only up to $\GL_2(\Z)$-equivalence. Pick $P_{i+1}\in\Gamma_i$ such that $\abs{\bdry{P_{i+1}}\cap N}=\min{\abs{\bdry{Q}\cap N}\mid Q\in\Gamma_i}$. If $\abs{\bdry{P_i}\cap N}=\abs{\bdry{P_{i+1}}\cap N}$ we stop, and $P_i$ is minimal. Notice that this process must terminate in a finite number of steps.

\begin{lemma}[Characterisation of minimality]\label{lem:characterisation}
Let $P\subset\NQ$ be a Fano polygon. The following are equivalent:
\begin{enumerate}
\item
$P$ is minimal;
\item
$\abs{\bdry{P}\cap N}\leq\abs{\bdry{Q}\cap N}$ for every mutation $Q:=\mut_w(P,F)$;
\item\label{item:characterisation_intr}
$\abs{\intr{P}\cap N}\leq\abs{\intr{Q}\cap N}$ for every mutation $Q:=\mut_w(P,F)$;
\item\label{item:characterisation_vol}
$\Vol{P}\leq\Vol{Q}$ for every mutation $Q:=\mut_w(P,F)$;
\item\label{item:characterisation_heights}
$r_1+\cdots+r_n\leq s_1+\cdots+s_n$ for every mutation $Q:=\mut_w(P,F)$, where the $r_i$ (respectively $s_i$) are the heights of the primitive $T$-cones associated with $P$ (respectively $Q$).
\end{enumerate}
\end{lemma}

\begin{proof}
Let $C$ be an arbitrary cone (not necessarily a $T$-cone) corresponding to the cyclic quotient singularity $\frac{1}{R}(a,b)$. Let  $\rho_1,\rho_2\in N$ be the primitive lattice vectors generating the rays of $C$. Recall that $k=\gcd{R,a+b}$ is the width of the line segment $\rho_1\rho_2$, and that $r=R/\gcd{R,a+b}$ is the height of $\rho_1\rho_2$. Set $D:=\sconv{\rho_1,\rho_2,\orig}$. Since $R=k+2\abs{\intr{D}\cap N}$ we have that:
$$\abs{\intr{D}\cap N}=\frac{R-k}{2}=\frac{k(r-1)}{2}.$$
If $C$ is a primitive $T$-cone then $r=k$.

Let $P$ have singularity content $(n,\cB)$, and let $r_1,\ldots,r_n$ be the heights of the primitive $T$-cones. Then the number of boundary points is
\begin{equation}\label{eq:minimal_bdry}
\abs{\bdry{P}\cap N}=\sum_{i=1}^nr_i+\sum_\cB(\abs{\bdry{D}\cap N}-1),
\end{equation}
where $\sum_\cB(\abs{\bdry{D}\cap N}-1)$ is the contribution arising from the basket $\cB$, and the number of interior points is given by
\begin{equation}\label{eq:minimal_intr}
\abs{\intr{P}\cap N}=1+\frac{1}{2}\sum_{i=1}^nr_i(r_i-1)+\sum_\cB\abs{\intr{D}\cap N},
\end{equation}
where $\sum_\cB\abs{\intr{D}\cap N}$ is the contribution arising from the basket $\cB$. Notice that the values of both $\sum_\cB(\abs{\bdry{D}\cap N}-1)$ and $\sum_\cB\abs{\intr{D}\cap N}$ are fixed under mutation. By applying Pick's formula we obtain
\begin{equation}\label{eq:minimal_vol}
\Vol{P}=\abs{\bdry{P}\cap N}+2\abs{\intr{P}\cap N}-2=\sum_{i=1}^nr_i^2+B
\end{equation}
where $B:=\sum_\cB(\abs{\bdry{D}\cap N}-1)+2\sum_\cB\abs{\intr{D}\cap N}$ is a constant under mutation. Mutation can change the value of only one $r_i$ at a time, hence equations~\eqref{eq:minimal_bdry},~\eqref{eq:minimal_intr}, and~\eqref{eq:minimal_vol} are all locally minimal with respect to mutation if and only if $r_1+\ldots+r_n$ is locally minimal with respect to mutation.
\end{proof}

\begin{remark}\label{rem:minimality_like_diophantine}
Recall that mutations between Fano triangles are characterised in terms of solutions to a Diophantine equation~\cite{AK13}. Every solution can be obtained from a minimal solution -- a solution $(a,b,c)\in\Z_{>0}^3$ whose sum $a+b+c$ is minimal -- and a minimal solution corresponds to a triangle with smallest volume~\cite[Lemma~3.16]{AK13}. By Lemma~\ref{lem:characterisation}~\eqref{item:characterisation_vol} we see that the notion of minimality introduced in Definition~\ref{defn:minimaility} above can be viewed as a generalisation of the concept of minimal solution.
\end{remark}

\begin{example}\label{eg:reflexive_are_minimal}
Any reflexive polygon is minimal by Lemma~\ref{lem:characterisation}~\eqref{item:characterisation_intr}, since $\abs{\intr{P}\cap N}=1$. In particular this gives us examples of mutation-equivalent polygons $P_1\not\cong P_2$, both of which are minimal. For example, one could take $P_1=\sconv{(\pm1,0),(0,\pm1)}$ and $P_2=\sconv{(1,0),(0,1),(-1,-2)}$ (the polygons associated with $\Proj^1\times\Proj^1$ and $\Proj(1,1,2)$, respectively).
\end{example}

\noindent
An immediate consequence of Lemma~\ref{lem:characterisation}~\eqref{item:characterisation_heights} is the following:

\begin{cor}\label{cor:minimise_heights}
Let $P\subset\NQ$ be a Fano polygon. For each edge $E$ of $P$ let $w_E\in M$ denote the corresponding primitive inner normal vector and let $k_E$ denote the width. $P$ is minimal if and only if $\abs{\hmin}\leq\hmax$ for each $w_E$ such that $k_E\geq\abs{\hmin}$.
\end{cor}

\begin{example}\label{eg:centrally_symmetric}
Any centrally symmetric polygon (that is, any polygon $P$ satisfying $v\in P$ if and only if $-v\in P$) is minimal: for any primitive inner normal vector $w\in M$, $\abs{\hmin}=\hmax$.
\end{example}

\begin{cor}\label{cor:triangle_characterisation}
Let $P:=\sconv{v_0,v_1,v_2}$ be a Fano triangle with residual basket $\cB=\emptyset$. Then $P$ is minimal if and only if $v_0+v_1+v_2\in P$.
\end{cor}

\begin{proof}
Set $\overline{v}:=v_0+v_1+v_2\in N$. Let $E$ be an edge of $P$ and let $w\in M$ be the corresponding primitive inner normal vector. Then $w(\overline{v})=w(v_0)+w(v_1)+w(v_2)=2\hmin+\hmax$, and $\abs{\hmin}\leq\hmax$ if and only if $\overline{v}$ lies in the half-space $H_E:=\{v\in\NQ\mid w(v)\geq\hmin\}$. By Corollary~\ref{cor:minimise_heights} we have that $P$ is minimal if and only if $\overline{v}\in\bigcap H_E$, where the intersection is taken over all edges $E$ of $P$. The result follows.
\end{proof}

%-------------------------------------------------------------------------------
\section{Minimal Fano polygons with only $T$-singularities}\label{sec:T_singular_polygons}
%-------------------------------------------------------------------------------
In Theorem~\ref{thm:T-sing_minimals} below we classify all minimal Fano polygons with residual basket $\cB=\emptyset$. Conjecture~\hyperlink{conj:A}{A} tells us that the mutation-equivalence classes should correspond to the ten qG-deformation classes of smooth del~Pezzo surfaces, and in Theorem~\ref{thm:T-sing_mutation_classes} we find that this is indeed what happens. In particular, in the case when $X_P$ is a rank-one toric del~Pezzo surface with only $T$-singularities we recover the results of Hacking--Prokhorov~\cite[Theorem~4.1]{HP10}. Perling has also studied these surfaces from the viewpoint of mutations of exceptional collections~\cite{P13}.

\begin{definition}
Let $P\subset\NQ$ be a Fano polygon with vertex set denoted by $\V{P}$ and edge set denoted by $\F{P}$. For each edge $E\in\F{P}$ of $P$ the unique primitive lattice point in the dual lattice $M:=\Hom{N,\Z}$ defining an inner normal of $E$ is denoted by $w_E$. The integer $\rE:=-w_E(E)$, equal to the height of $E$ above the origin $\orig$, is called the \emph{local index} of $E$ (or the \emph{Gorenstein index} of the singularity associated with $\cone{E}$). The \emph{maximum local index} is the maximum local index of all edges of $P$:
$$\mP:=\max{\rE\mid E\in\F{P}}.$$
\end{definition}

\begin{lemma}\label{lem:triangle_degree}
Let $P=\sconv{v_0,v_1,v_2}\subset\NQ$ be a Fano triangle, and let $(\lambda_0,\lambda_1,\lambda_2)\in\Z^3_{>0}$ be pairwise coprime weights such that $\lambda_0v_0+\lambda_1v_1+\lambda_2v_2=\orig$. Then:
$$\Vol{P}\cdot\Vol{\dual{P}}=\frac{(\lambda_0+\lambda_1+\lambda_2)^3}{\lambda_0\lambda_1\lambda_2}.$$
\end{lemma}
\begin{proof}
Recall~\cite{Kas08b} that the Fano triangle $P$ corresponds to some fake weighted projective space $X_P=\Proj(\lambda_0,\lambda_1,\lambda_2)/G$, where $G=N/N'$ is the quotient of $N$ by the sublattice $N'=v_0\cdot\Z+v_1\cdot\Z+v_2\cdot\Z$ generated by the vertices of $P$. The order $\abs{G}$, or equivalently the index $[N:N']$ of the sublattice $N'$, is called the \emph{multiplicity} of $P$, and denoted by $\mult{P}$. Let $Q$ be the Fano triangle associated with weighted projective space $X_Q=\Proj(\lambda_0,\lambda_1,\lambda_2)$. Then
$$\Vol{Q}=\lambda_0+\lambda_1+\lambda_2\quad\text{ and }\quad
\Vol{\dual{Q}}=\frac{(\lambda_0+\lambda_1+\lambda_2)^2}{\lambda_0\lambda_1\lambda_2},$$
where the second value is simply the degree of the weighted projective space $X_Q$. But $Q$ is $\GL_2(\Z)$-equivalent to the restriction of $P$ to $N'$. Hence $\Vol{P}=\mult{P}\cdot\Vol{Q}$ and:
$$\Vol{\dual{P}}=\frac{1}{\mult{P}}\cdot\Vol{\dual{Q}}=\frac{(\lambda_0+\lambda_1+\lambda_2)^3}{\lambda_0\lambda_1\lambda_2}\cdot\frac{1}{\Vol{P}}.$$
\end{proof}

\begin{lemma}[\protect{\cite[Lemma~2.1]{KKN08}}]\label{lem:edge_projection}
Let $P\subset\NQ$ be a Fano polygon and let $E$ be an edge with local index equal to the maximum local index of $P$, $\rE=\mP$. Assume that there exists a non-vertex lattice point $v\in\cone{E}$ with $w_E(v)=-1$. For every lattice point $v'\in P\cap N\setminus E$ we have that $v+v'\in P\cap N$. In particular, if $E'$ is an edge of $P$ that is not parallel to $v$, then $|E' \cap N|\le |E \cap N|$.
\end{lemma}
%-------------------------------------------------------------------------------
\begin{table}[htdp]
\centering
\caption{The minimal Fano triangles $P\subset\NQ$ with singularity content $(n,\emptyset)$, vertices $\V{P}$, and weights $(\lambda_0,\lambda_1,\lambda_2)$, up to the action of $\GL_2(\Z)$. The corresponding toric varieties $X_P$ and degrees $(-K_{X_P})^2=12-n$ are also given. See also Figure~\ref{fig:triangle_smoothable}.}
\label{tab:triangle_smoothable}
\begin{tabular}{rlcccccc}
\toprule
\multicolumn{1}{c}{$P$}&\multicolumn{1}{c}{$\V{P}$}&$\lambda_0$&$\lambda_1$&$\lambda_2$&$X_P$&$n$&$(-K_{X_P})^2$\\
\cmidrule(lr){1-1} \cmidrule(lr){2-2} \cmidrule(lr){3-5} \cmidrule(lr){6-6} \cmidrule(lr){7-8}
\oddrow \padding \hyperlink{fig:RpolyP}{$\RpolyP$}&\gap $(1,0),(0,1),(-1,-1)$&$1$&$1$&$1$&$\Proj^2$&$3$&$9$\\
\oddrow \padding \hyperlink{fig:RpolyL}{$\RpolyL$}&\gap $(1,1),(-2,1),(1,-2)$&$1$&$1$&$1$&$\Proj^2/(\Z/3)$&$9$&$3$\\
\oddrow \padding \hyperlink{fig:TpolyH}{$\TpolyH$}&\gap $(1,3),(-2,3),(1,-6)$&$1$&$1$&$1$&$\Proj^2/(\Z/9)$&$11$&$1$\\
\cmidrule(lr){1-1} \cmidrule(lr){2-2} \cmidrule(lr){3-5} \cmidrule(lr){6-6} \cmidrule(lr){7-8}
\evnrow \padding \hyperlink{fig:RpolyO}{$\RpolyO$}&\gap $(1,1),(-1,1),(0,-1)$&$1$&$1$&$2$&$\Proj(1,1,2)$&$4$&$8$\\
\evnrow \padding \hyperlink{fig:RpolyM}{$\RpolyM$}&\gap $(1,1),(-1,1),(1,-3)$&$1$&$1$&$2$&$\Proj(1,1,2)/(\Z/2)$&$8$&$4$\\
\evnrow \padding \hyperlink{fig:TpolyE}{$\TpolyE$}&\gap $(1,2),(-1,2),(1,-6)$&$1$&$1$&$2$&$\Proj(1,1,2)/(\Z/4)$&$10$&$2$\\
\evnrow \padding \hyperlink{fig:TpolyA}{$\TpolyA$}&\gap $(3,2),(-1,2),(-1,-2)$&$1$&$1$&$2$&$\Proj(1,1,2)/(\Z/4)$&$10$&$2$\\
\evnrow \padding \hyperlink{fig:TpolyG}{$\TpolyG$}&\gap $(1,4),(-3,4),(1,-4)$&$1$&$1$&$2$&$\Proj(1,1,2)/(\Z/8)$&$11$&$1$\\
\cmidrule(lr){1-1} \cmidrule(lr){2-2} \cmidrule(lr){3-5} \cmidrule(lr){6-6} \cmidrule(lr){7-8}
\oddrow \padding \hyperlink{fig:RpolyN}{$\RpolyN$}&\gap $(1,1),(-1,1),(1,-2)$&$1$&$2$&$3$&$\Proj(1,2,3)$&$6$&$6$\\
\oddrow \padding \hyperlink{fig:TpolyD}{$\TpolyD$}&\gap $(1,2),(-1,2),(1,-4)$&$1$&$2$&$3$&$\Proj(1,2,3)/(\Z/2)$&$9$&$3$\\
\oddrow \padding \hyperlink{fig:TpolyI}{$\TpolyI$}&\gap $(1,3),(-2,3),(1,-3)$&$1$&$2$&$3$&$\Proj(1,2,3)/(\Z/3)$&$10$&$2$\\
\oddrow \padding \hyperlink{fig:TpolyB}{$\TpolyB$}&\gap $(5,3),(-1,3),(-1,-3)$&$1$&$2$&$3$&$\Proj(1,2,3)/(\Z/6)$&$11$&$1$\\
\cmidrule(lr){1-1} \cmidrule(lr){2-2} \cmidrule(lr){3-5} \cmidrule(lr){6-6} \cmidrule(lr){7-8}
\evnrow \padding \hyperlink{fig:TpolyC}{$\TpolyC$}&\gap $(1,2),(-1,2),(1,-3)$&$1$&$4$&$5$&$\Proj(1,4,5)$&$7$&$5$\\
\evnrow \padding \hyperlink{fig:TpolyF}{$\TpolyF$}&\gap $(2,5),(-3,5),(2,-5)$&$1$&$4$&$5$&$\Proj(1,4,5)/(\Z/5)$&$11$&$1$\\
\bottomrule
\end{tabular}
\end{table}
%-------------------------------------------------------------------------------
\begin{table}[htdp]
\centering
\caption{The minimal Fano $m$-gons, $m\geq 4$, $P\subset\NQ$ with singularity content $(n,\emptyset)$ and vertices $\V{P}$, up to the action of $\GL_2(\Z)$. The degrees $(-K_{X_P})^2=12-n$ of the corresponding toric varieties are also given. See also Figure~\ref{fig:polygon_smoothable}.}
\label{tab:polygon_smoothable}
\begin{tabular}{rlcccccc}
\toprule
\multicolumn{1}{c}{$P$}&\multicolumn{1}{c}{$\V{P}$}&$n$&$(-K_{X_P})^2$\\
\cmidrule(lr){1-1} \cmidrule(lr){2-2} \cmidrule(lr){3-4}
\oddrow \padding \hyperlink{fig:PpolyJ}{$\PpolyJ$}&\gap $(1,2),(-1,2),(-1,-2),(1,-2)$&10&2\\
\oddrow \padding \hyperlink{fig:PpolyG}{$\PpolyG$}&\gap $(1,2),(-1,2),(-1,-1),(1,-3)$&10&2\\
\oddrow \padding \hyperlink{fig:PpolyH}{$\PpolyH$}&\gap $(1,2),(-1,2),(-1,0),(1,-4)$&10&2\\
\oddrow \padding \hyperlink{fig:PpolyI}{$\PpolyI$}&\gap $(1,2),(-1,2),(-1,1),(1,-5)$&10&2\\
\cmidrule(lr){1-1} \cmidrule(lr){2-2} \cmidrule(lr){3-4}
\evnrow \padding \hyperlink{fig:PpolyB}{$\PpolyB$}&\gap $(1,2),(-1,2),(-1,0),(1,-2)$&9&3\\
\evnrow \padding \hyperlink{fig:PpolyE}{$\PpolyE$}&\gap $(1,2),(-1,2),(-1,1),(1,-3)$&9&3\\
\cmidrule(lr){1-1} \cmidrule(lr){2-2} \cmidrule(lr){3-4}
\oddrow \padding \hyperlink{fig:RpolyE}{$\RpolyE$}&\gap $(1,1),(-1,1),(-1,0),(1,-2)$&8&4\\
\oddrow \padding \hyperlink{fig:RpolyF}{$\RpolyF$}&\gap $(1,1),(-1,1),(-1,-1),(1,-1)$&8&4\\
\oddrow \padding \hyperlink{fig:PpolyF}{$\PpolyF$}&\gap $(1,2),(-1,2),(0,-1),(1,-3)$&8&4\\
\oddrow \padding \hyperlink{fig:PpolyD}{$\PpolyD$}&\gap $(1,2),(-1,2),(-1,1),(0,-1),(1,-2)$&8&4\\
\cmidrule(lr){1-1} \cmidrule(lr){2-2} \cmidrule(lr){3-4}
\evnrow \padding \hyperlink{fig:RpolyG}{$\RpolyG$}&\gap $(1,1),(-1,1),(0,-1),(1,-2)$&7&5\\
\evnrow \padding \hyperlink{fig:PpolyA}{$\PpolyA$}&\gap $(1,2),(-1,2),(-1,1),(1,-2)$&7&5\\
\evnrow \padding \hyperlink{fig:PpolyC}{$\PpolyC$}&\gap $(1,2),(-1,2),(0,-1),(1,-2)$&7&5\\
\evnrow \padding \hyperlink{fig:RpolyB}{$\RpolyB$}&\gap $(1,1),(-1,1),(-1,0),(0,-1),(1,-1)$&7&5\\
\cmidrule(lr){1-1} \cmidrule(lr){2-2} \cmidrule(lr){3-4}
\oddrow \padding \hyperlink{fig:RpolyH}{$\RpolyH$}&\gap $(1,1),(-1,1),(-1,0),(1,-1)$&6&6\\
\oddrow \padding \hyperlink{fig:RpolyC}{$\RpolyC$}&\gap $(1,0),(1,1),(-1,1),(-1,0),(0,-1)$&6&6\\
\oddrow \padding \hyperlink{fig:RpolyA}{$\RpolyA$}&\gap $(1,0),(1,1),(0,1),(-1,0),(-1,-1),(0,-1)$&6&6\\
\cmidrule(lr){1-1} \cmidrule(lr){2-2} \cmidrule(lr){3-4}
\evnrow \padding \hyperlink{fig:RpolyI}{$\RpolyI$}&\gap $(1,0),(1,1),(-1,1),(0,-1)$&5&7\\
\evnrow \padding \hyperlink{fig:RpolyD}{$\RpolyD$}&\gap $(1,0),(1,1),(0,1),(-1,-1),(0,-1)$&5&7\\
\cmidrule(lr){1-1} \cmidrule(lr){2-2} \cmidrule(lr){3-4}
\oddrow \padding \hyperlink{fig:RpolyJ}{$\RpolyJ$}&\gap $(1,0),(0,1),(-1,-1),(0,-1)$&4&8\\
\oddrow \padding \hyperlink{fig:RpolyK}{$\RpolyK$}&\gap $(1,0),(0,1),(-1,0),(0,-1)$&4&8\\
\bottomrule
\end{tabular}
\end{table}
%-------------------------------------------------------------------------------
\begin{figure}[htdp]
\caption{The minimal Fano triangles $P\subset\NQ$ with singularity content $(n,\emptyset)$, up to the action of $\GL_2(\Z)$. See also Table~\ref{tab:triangle_smoothable}.}
\label{fig:triangle_smoothable}
\centering
\renewcommand{\arraystretch}{0.8}
\vspace{0.5em}
\small
\begin{tabular}{cccccc}
\hypertarget{fig:RpolyP}{}$\RpolyP$&
\hypertarget{fig:RpolyL}{}$\RpolyL$&
\hypertarget{fig:TpolyH}{}$\TpolyH$&
\hypertarget{fig:RpolyO}{}$\RpolyO$&
\hypertarget{fig:RpolyM}{}$\RpolyM$&
\hypertarget{fig:TpolyE}{}$\TpolyE$\\
\raisebox{109pt}{\includegraphics[scale=0.6]{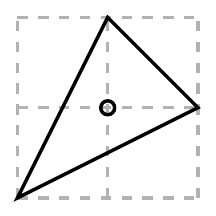}}&% original height = 62pt
\raisebox{94pt}{\includegraphics[scale=0.6]{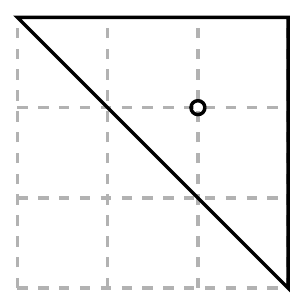}}&% original height = 88pt
\raisebox{0pt}{\includegraphics[scale=0.6]{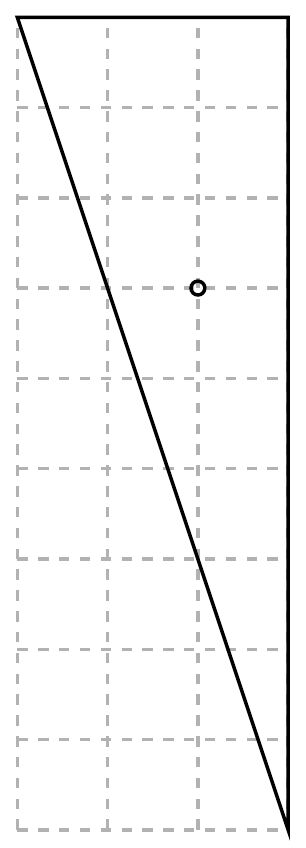}}&% original height = 244pt
\raisebox{109pt}{\includegraphics[scale=0.6]{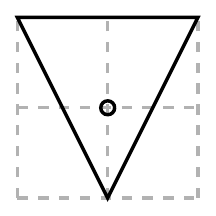}}&% original height = 62pt
\raisebox{78pt}{\includegraphics[scale=0.6]{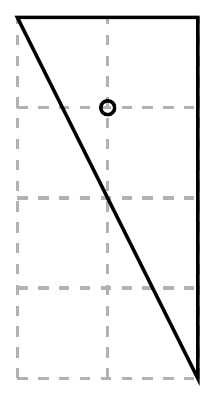}}&% original height = 114pt
\raisebox{16pt}{\includegraphics[scale=0.6]{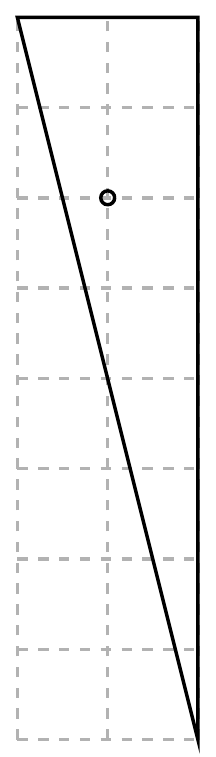}}\\% original height = 218pt
\end{tabular}
\vgap

\begin{tabular}{ccccc}
\hypertarget{fig:TpolyA}{}$\TpolyA$&
\hypertarget{fig:TpolyG}{}$\TpolyG$&
\hypertarget{fig:RpolyN}{}$\RpolyN$&
\hypertarget{fig:TpolyD}{}$\TpolyD$&
\hypertarget{fig:TpolyI}{}$\TpolyI$\\
\raisebox{62pt}{\includegraphics[scale=0.6]{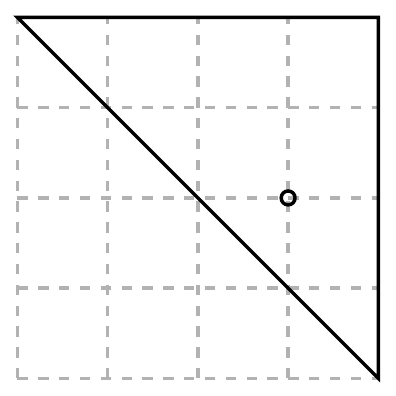}}&% original height = 114pt
\raisebox{0pt}{\includegraphics[scale=0.6]{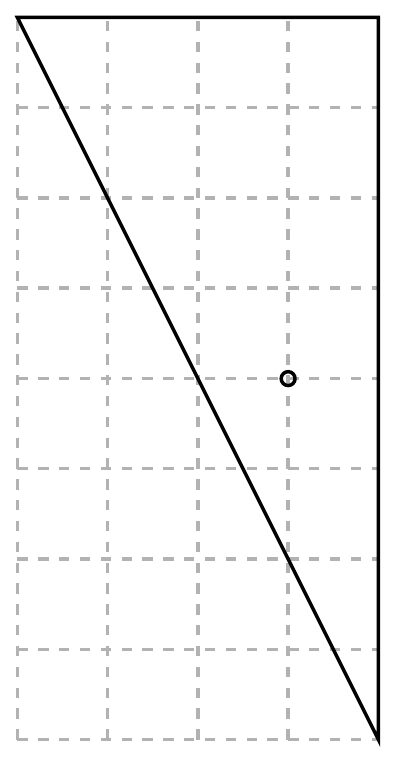}}&% original height = 218pt
\raisebox{78pt}{\includegraphics[scale=0.6]{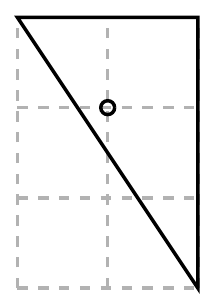}}&% original height = 88pt
\raisebox{31pt}{\includegraphics[scale=0.6]{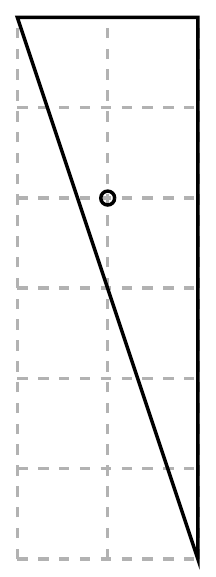}}&% original height = 166pt
\raisebox{31pt}{\includegraphics[scale=0.6]{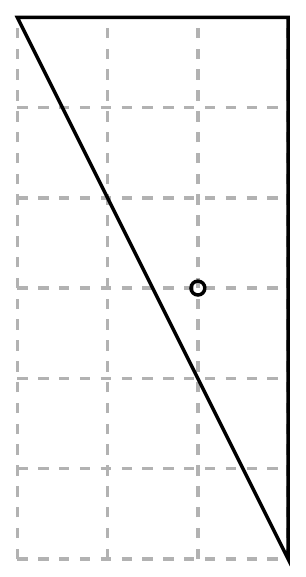}}\\% original height = 166pt
\end{tabular}
\vgap

\begin{tabular}{ccc}
\hypertarget{fig:TpolyB}{}$\TpolyB$&
\hypertarget{fig:TpolyC}{}$\TpolyC$&
\hypertarget{fig:TpolyF}{}$\TpolyF$\\
\raisebox{62pt}{\includegraphics[scale=0.6]{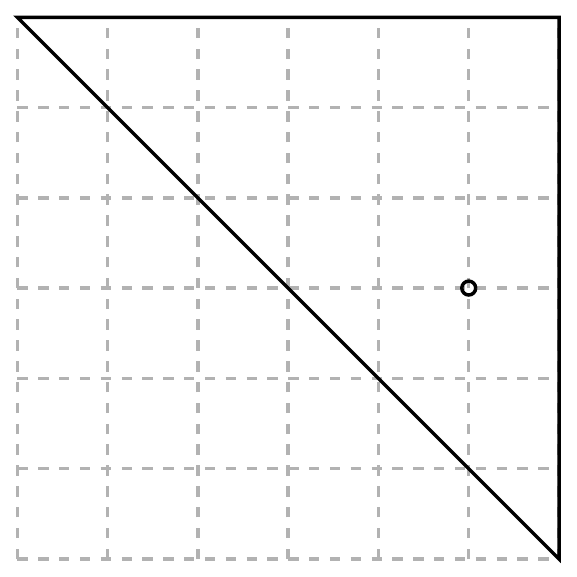}}&% original height = 166pt
\raisebox{78pt}{\includegraphics[scale=0.6]{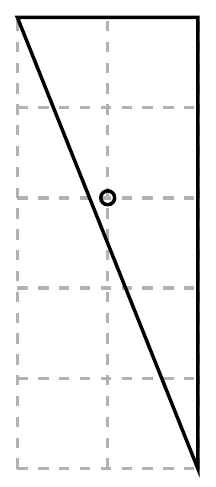}}&% original height = 140pt
\raisebox{0pt}{\includegraphics[scale=0.6]{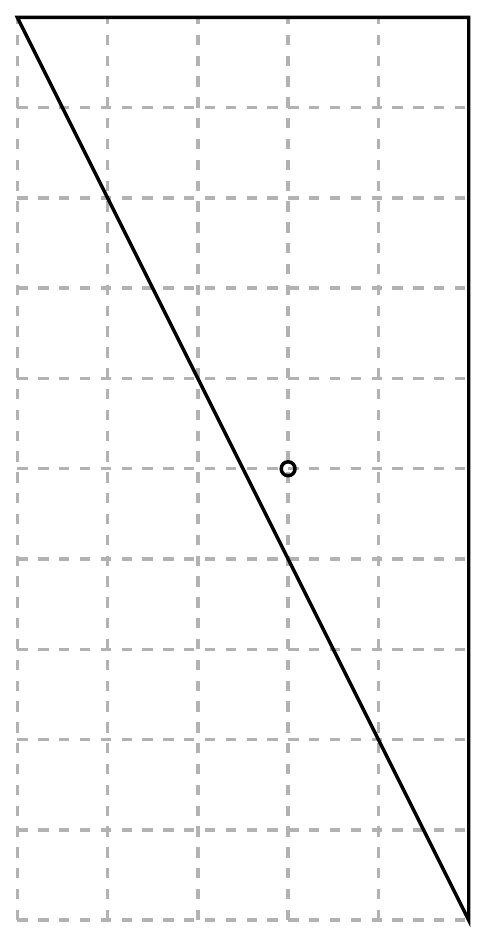}}\\% original height = 270pt
\end{tabular}
\end{figure}
%-------------------------------------------------------------------------------
\begin{figure}[htdp]
\caption{The minimal Fano $m$-gons, $m\geq 4$, $P\subset\NQ$ with singularity content $(n,\emptyset)$, up to the action of $\GL_2(\Z)$. See also Table~\ref{tab:polygon_smoothable}.}
\label{fig:polygon_smoothable}
\centering
\vspace{0.5em}
\renewcommand{\arraystretch}{0.8}
\small
\begin{tabular}{cccccccc}
\hypertarget{fig:PpolyJ}{}$\PpolyJ$&
\hypertarget{fig:PpolyG}{}$\PpolyG$&
\hypertarget{fig:PpolyH}{}$\PpolyH$&
\hypertarget{fig:PpolyI}{}$\PpolyI$&
\hypertarget{fig:PpolyB}{}$\PpolyB$&
\hypertarget{fig:PpolyE}{}$\PpolyE$&
\hypertarget{fig:RpolyE}{}$\RpolyE$&
\hypertarget{fig:RpolyF}{}$\RpolyF$\\
\raisebox{47pt}{\includegraphics[scale=0.6]{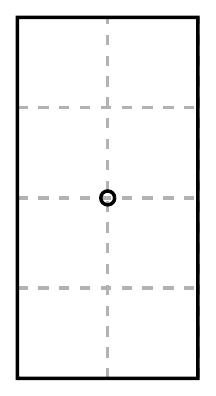}}&% original height = 114pt
\raisebox{31pt}{\includegraphics[scale=0.6]{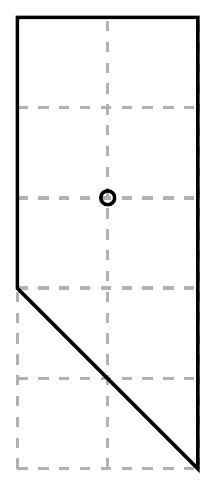}}&% original height = 140pt
\raisebox{16pt}{\includegraphics[scale=0.6]{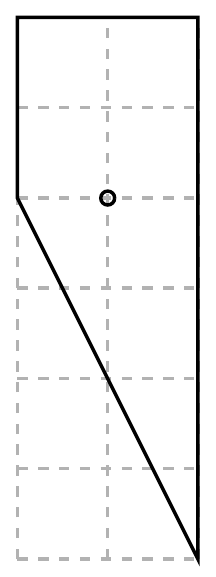}}&% original height = 166pt
\raisebox{0pt}{\includegraphics[scale=0.6]{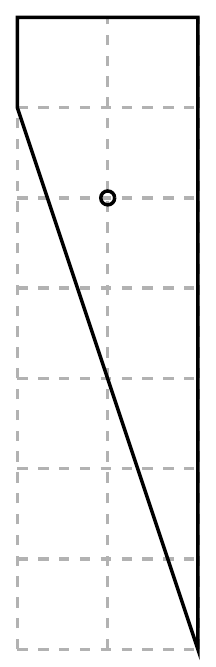}}&% original height = 192pt
\raisebox{47pt}{\includegraphics[scale=0.6]{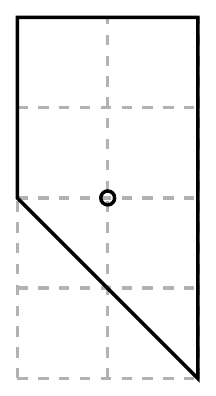}}&% original height = 114pt
\raisebox{31pt}{\includegraphics[scale=0.6]{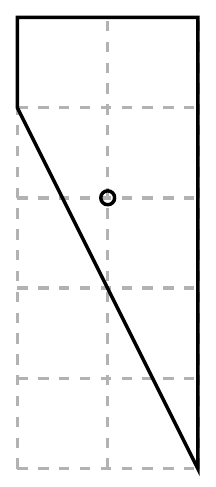}}&% original height = 140pt
\raisebox{62pt}{\includegraphics[scale=0.6]{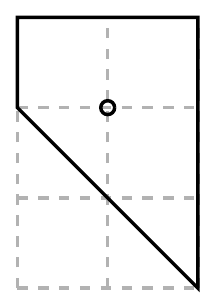}}&% original height = 88pt
\raisebox{78pt}{\includegraphics[scale=0.6]{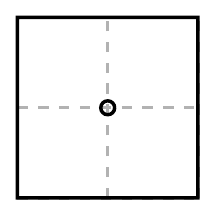}}\\% original height = 62pt
\end{tabular}
\vgap

\begin{tabular}{cccccccc}
\hypertarget{fig:PpolyF}{}$\PpolyF$&
\hypertarget{fig:PpolyD}{}$\PpolyD$&
\hypertarget{fig:RpolyG}{}$\RpolyG$&
\hypertarget{fig:PpolyA}{}$\PpolyA$&
\hypertarget{fig:PpolyC}{}$\PpolyC$&
\hypertarget{fig:RpolyB}{}$\RpolyB$&
\hypertarget{fig:RpolyH}{}$\RpolyH$&
\hypertarget{fig:RpolyC}{}$\RpolyC$\\
\raisebox{0pt}{\includegraphics[scale=0.6]{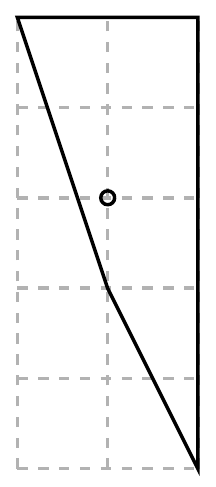}}&% original height = 140pt
\raisebox{16pt}{\includegraphics[scale=0.6]{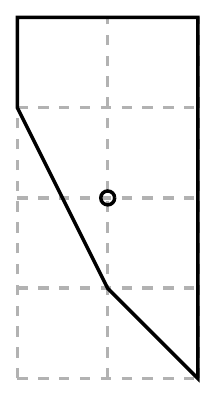}}&% original height = 114pt
\raisebox{31pt}{\includegraphics[scale=0.6]{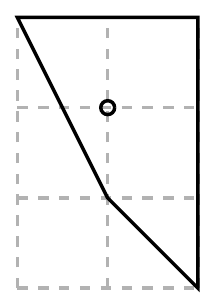}}&% original height = 88pt
\raisebox{16pt}{\includegraphics[scale=0.6]{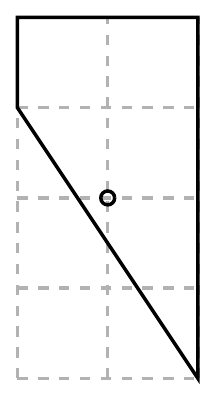}}&% original height = 114pt
\raisebox{16pt}{\includegraphics[scale=0.6]{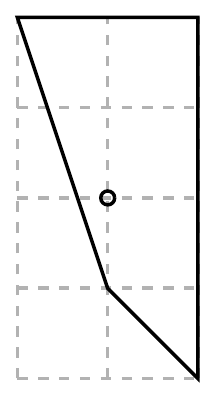}}&% original height = 114pt
\raisebox{47pt}{\includegraphics[scale=0.6]{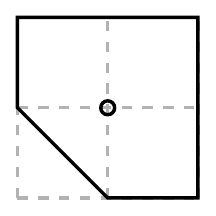}}&% original height = 62pt
\raisebox{47pt}{\includegraphics[scale=0.6]{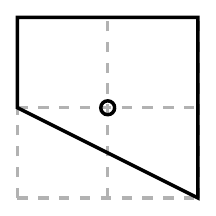}}&% original height = 62pt
\raisebox{47pt}{\includegraphics[scale=0.6]{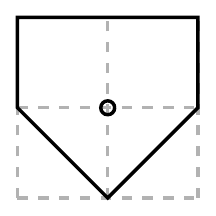}}\\% original height = 62pt
\end{tabular}
\vgap

\begin{tabular}{ccccc}
\hypertarget{fig:RpolyA}{}$\RpolyA$&
\hypertarget{fig:RpolyI}{}$\RpolyI$&
\hypertarget{fig:RpolyD}{}$\RpolyD$&
\hypertarget{fig:RpolyJ}{}$\RpolyJ$&
\hypertarget{fig:RpolyK}{}$\RpolyK$\\
\raisebox{0pt}{\includegraphics[scale=0.6]{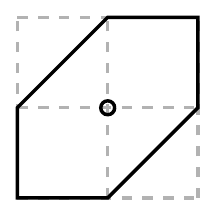}}&% original height = 62pt
\raisebox{0pt}{\includegraphics[scale=0.6]{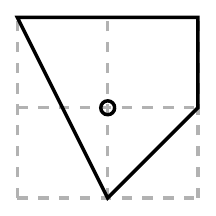}}&% original height = 62pt
\raisebox{0pt}{\includegraphics[scale=0.6]{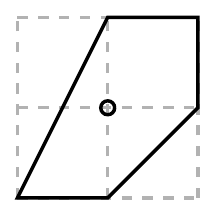}}&% original height = 62pt
\raisebox{0pt}{\includegraphics[scale=0.6]{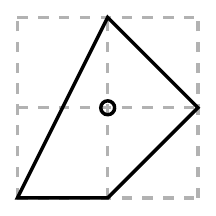}}&% original height = 62pt
\raisebox{0pt}{\includegraphics[scale=0.6]{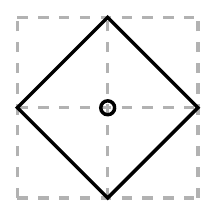}}\\% original height = 62pt
\end{tabular}
\end{figure}
%-------------------------------------------------------------------------------
\begin{thm}\label{thm:T-sing_minimals}
Let $P\subset\NQ$ be a minimal Fano polygon with residual basket $\cB=\emptyset$. Then, up to $\GL_2(\Z)$-equivalence, $P$ is one of the following $35$ polygons:
\begin{enumerate}
\item\label{item:T-sing_minimals_reflexive_triangle}
the five reflexive triangles $\RpolyP,\ldots,\RpolyN$ in Table~\ref{tab:triangle_smoothable};
\item\label{item:T-sing_minimals_reflexive_exceptional}
the eleven reflexive polygons $\RpolyE,\ldots,\RpolyK$ in Table~\ref{tab:polygon_smoothable};
\item\label{item:T-sing_minimals_triangle}
the nine non-reflexive triangles $\TpolyH,\ldots,\TpolyF$ in Table~\ref{tab:triangle_smoothable};
\item\label{item:T-sing_minimals_exceptional}
the ten non-reflexive polygons $\PpolyJ,\ldots,\PpolyC$ in Table~\ref{tab:polygon_smoothable}.
\end{enumerate}
\end{thm}

\begin{proof}
We consider the cases when $P$ is reflexive, that is $\mP=1$, and when $P$ is not reflexive, that is $\mP>1$, separately.

\proofsection{$\mP=1$}
Every reflexive polygon is minimal since $\abs{\intr{P}\cap N}=1$. The classification of the reflexive polygons is well-known: up to $\GL_2(\Z)$-equivalence there are sixteen polygons~\cite{Bat85,Rab89}, of which five are triangles. This proves~\eqref{item:T-sing_minimals_reflexive_triangle} and~\eqref{item:T-sing_minimals_reflexive_exceptional}.

\proofsection{$\mP>1$}
Let $E\in\F{P}$ be an edge of maximum local index $\rE=\mP>1$. By assumption the width of $E$ is some multiple $k\in\Z_{>0}$ of $\rE$, that is, $\abs{E\cap N}=k\rE+1$. Since the cone $C=\cone{E}$ is a union of $k$ primitive $T$-cones, each of which contains exactly one point at height $-1$ with respect to $w_E$, there exist $k$ distinct points $v_1,\ldots,v_k\in\intr{C}\cap N$ such that $w_E(v_i)=-1$. After suitable change of basis we can insist that $E=\sconv{(-a,\rE),(b,\rE)}$ where $a,b>0$, $a<\rE$, and $a+b=k\rE$. Hence the point $\left((i-1)\rE,\rE\right)$ lies in the strict interior of $E$ and so, after possible reordering, $v_i=(i-1,1)$, for each $i=1,\ldots,k$.

Let $v\in\V{P}$ be a vertex of $P$ such that $w_E(v)=\max{w_E(v')\mid v'\in P}=\hmax$. Since $P$ is minimal by assumption, we have that $w_E(v)\geq\rE=\abs{\hmin}$ (Corollary~\ref{cor:minimise_heights}). By Lemma~\ref{lem:edge_projection} we have that $v$ is contained in the strip $E-v_1\cdot\Z_{>0}$. Hence we can write $v=(\alpha,-\beta)$ for some $\beta\geq\rE$ and $-a\leq\alpha\leq b$. If $k>1$ then we have $v\in E-v_k\cdot\Z_{>0}$ and so
\begin{equation}\label{eq:point_on_bottom_edge}
-a\leq\alpha\leq b-(k-1)(\rE+\beta)\leq b-2(k-1)\rE=-a-(k-2)\rE.
\end{equation}
We conclude that $k\leq 2$.

\proofsection{$k=2$}
Let us first consider the case when $k=2$. By~\eqref{eq:point_on_bottom_edge} $P$ is a triangle given by
$$P=\sconv{(-a,\rE),(-a+2\rE,\rE),(-a,-\rE)},\quad\text{ where }0<a<\rE, \gcd{a,\rE}=1.$$
Let $E'\in\F{P}$ be the edge with vertices $(-a,\rE)$ and $(-a,-\rE)$. Since the corresponding cone is of class $T$, we have that $a\divides 2\rE$, and so $a=1$ or $2$. First consider the case $a=1$. The edge with vertices $(-1,-\rE)$ and $(2\rE-1,\rE)$ is of width $2\rE$ and height $\rE-1$, hence $\rE-1\divides 2\rE$ and so $\rE=2$ or $3$. This gives the two triangles $\TpolyA$ and $\TpolyB$ in Table~\ref{tab:triangle_smoothable}. Similarly consider the case $a=2$. In this case we see that $\rE-2\divides 2\rE$ and so $\rE=3$, $4$, or $6$, however the final two possibilities are excluded because the vertices are non-primitive. This also gives the triangle $\TpolyB$.

\proofsection{$k=1$}
We now consider the case when $k=1$. This is subdivided into two cases depending on whether or not there exists an edge $E'\in\F{P}$ parallel to $E$.

\proofsection{Edge $E'$ parallel to $E$}
Let us assume that there exists a second point of $P$ at height $w_E(v)$; that is, that there exists an edge $E'\in\F{P}$ such that $w_E(E')=w_E(v)$, so that $E$ and $E'$ are parallel. By minimality we see that $r_{E'}=\rE$, and by Lemma~\ref{lem:edge_projection} we have that $\abs{E'\cap N}\leq\abs{E\cap N}=\rE+1$. Recalling that $\cone{E'}$ is a $T$-cone, and hence $r_{E'}\divides\abs{E'\cap N}-1$, we see that $P$ is a rectangle:
$$P=\sconv{(-a,\rE),(-a+\rE,\rE),(-a,-\rE),(-a+\rE,-\rE)}.$$
Since $P$ is minimal we have that $a=2\rE$ and, by primitivity of the vertices, $a=1$, $\rE=2$, giving $\PpolyJ$ in Table~\ref{tab:polygon_smoothable}.

\proofsection{No edge parallel to $E$}
We are now in the situation where $v$ is the unique vertex satisfying $w_E(v)=\max{w_E(v')\mid v'\in P}$. We subdivide this into two cases depending on whether one of the edges $E'\in\F{P}$ with vertex $v$ is vertical.

\proofsection{$E'$ not vertical}
Let us assume that there is no vertical edge adjacent to $v$.  We consider the case $\alpha\ge 0$ (the case $\alpha<0$ being similar). By our assumption, we can choose an edge $E'$ adjacent to $v$ with $w_{E'}=(-\gamma,\delta)\in M$, where $\gamma,\delta\ge 1$. We have that $\rE\ge -w_{E'}(v)=\gamma\alpha+\delta\beta\ge\alpha+\beta\ge\beta\ge\rE$. Hence $(\alpha,-\beta)=(0,-\rE)$, $w_{E'}=(-\gamma,1)$, and $r_{E'}=\rE=\mP$. Exchanging the roles of $E$ and $E'$ we have that $E'$ is of width either $2\rE$, in which case we are in the case $k=2$ above, or of width $\rE$, in which case $\left(\rE,(\gamma-1)\rE\right)$ is a vertex of $P$. This is a contradiction ($\gamma\ne 1$ since $E'$ is not parallel to $E$, but then the vertex fails to be primitive).

\proofsection{$E'$ vertical}
The majority of cases arise when $v$ is contained in a vertical edge $E'\in\F{P}$. This edge necessarily contains one of the two vertices of $E$, and without loss of generality (since $a+b=\rE$) we may assume that $E'=\sconv{v,(b,\rE)}$. Hence $v=(b,\rE-jb)$ for some $j\in\Z_{>0}$. Minimality forces $\rE-jb\le -\rE$, so that
\begin{equation}\label{eq:ineq-b-min}
jb\ge2\rE.
\end{equation}
Moreover, minimality implies
\begin{equation}\label{eq:ineq-b-max}
2b\le\rE.
\end{equation}
In particular we see that $j\ge 4.$

\proofsection{$j=4$}
The case when $j=4$ is different from the cases when $j\geq 5$, and we deal with it now. Equations~\eqref{eq:ineq-b-min} and~\eqref{eq:ineq-b-max} gives that $2b=\rE$, and so by primitivity we have that $\rE=2$ and $b=1$. Hence $P$ is contained in the rectangle $[-1,1]\times [-2,2]$. Notice that the requirement that $\orig\in\intr{P}$ means that $P$ cannot be a triangle. We find $\PpolyB$, $\PpolyD$, $\PpolyA$, and $\PpolyC$ in Table~\ref{tab:polygon_smoothable}.

\proofsection{$j\ge5$}
Consider the triangle
$$T:=\sconv{(b-\rE,\rE),(b,\rE),(b,\rE-jb)}.$$
The edge joining $(b-\rE,\rE)$ and $(b,\rE-jb)$ has first coordinate equal to zero at the point $\big(0,\rE-jb+j\frac{b^2}{\rE}\big)$. Notice that
$$\rE-jb+j\frac{b^2}{\rE}=\left(b-\frac{\rE}{2}\right)\frac{jb}{\rE}+\left(\rE-\frac{jb}{2}\right)\leq 0,$$
where the inequality follows from~\eqref{eq:ineq-b-min} and~\eqref{eq:ineq-b-max}, and that $\rE-jb+j\frac{b^2}{\rE}=0$ implies that
$$b=\frac{\rE}{2}-\rE\sqrt{\frac{1}{4}-\frac{1}{j}}\in\Z,$$
which is impossible, since $j\geq 5$. Hence $\orig\in\intr{T}$ and $T$ is a Fano triangle. We have that $\Vol{T}=jb\rE$. Moreover, one can check that
$$\frac{1}{\gcd{j,\rE}}\left(jb^2,jb\rE-jb^2-\rE^2,\rE^2\right)\in\Z^3_{>0}$$
are pairwise coprime weights satisfying the conditions of Lemma~\ref{lem:triangle_degree}. Hence
$$\Vol{\dual{T}}=\frac{j}{jb\rE-jb^2-\rE^2}.$$
By Proposition~\ref{prop:sing_content_degree}, $\Vol{\dual{P}}\in\Z_{>0}$, and so $1\le\Vol{\dual{P}}\le\Vol{\dual{T}}$. Let $b=\rE q$, where $\frac{2}{j}\leq q\leq\frac{1}{2}$, by~\eqref{eq:ineq-b-min} and~\eqref{eq:ineq-b-max}, so that the lower bound on $\Vol{\dual{T}}$ gives:
\begin{equation}\label{eq:triangle}
\rE^2(-jq^2+jq-1)\le j.
\end{equation}
The quadratic in $q$ on the left-hand-side of~\eqref{eq:triangle} is strictly positive in the range $\frac{2}{j}\leq q\leq\frac{1}{2}$ and obtains its minimum value when $q=\frac{2}{j}$. Hence~\eqref{eq:triangle} gives us:
\begin{equation}\label{eq:triangle_bounds_l}
\rE^2\le\frac{j^2}{j-4}.
\end{equation}

Recall from Proposition~\ref{prop:sing_content_degree} that $\Vol{\dual{P}}=12-n\geq 1$, where $n$ is the total number of primitive $T$-cones spanned by the edges of $P$; equivalently,
$$n=\sum_{F\in\F{P}}\frac{\abs{F\cap N}-1}{r_F}.$$
Since $P$ must have at least three edges, each of which corresponds to a $T$-cone, and by construction we have that the top edge decomposes into a single primitive $T$-cone, and that the right-hand vertical edge $E'$ decomposes into $j\geq 5$ primitive $T$-cones, we conclude that $j\in\{5,\ldots,9\}$. From the inequalities~\eqref{eq:ineq-b-min}, \eqref{eq:ineq-b-max}, \eqref{eq:triangle_bounds_l}, and $\mP=\rE\ge 2$, along with the requirement that $\gcd{b,\rE}=1$, we obtain finitely many possibilities for the triple $(j,\rE,b)$, as recorded in Table~\ref{tab:posible_triples_T-sings}.

\begin{table}[htdp]
\centering
\caption{The possible values of $(j,\rE,b)$.}
\label{tab:posible_triples_T-sings}
\begin{tabular}{rgwgwgwgwgwgw}
\toprule
$j$&5&5&6&6&7&7&8&8&8&9&9&9\\
$\rE$&2&5&2&3&2&3&2&3&4&2&3&4\\
$b$&1&2&1&1&1&1&1&1&1&1&1&1\\
\bottomrule
\end{tabular}
\end{table}

\proofsection{Analysis of Table~\ref{tab:posible_triples_T-sings}}
First we consider the cases where $\rE=2$ and $b=1$. Here $P$ is contained in the rectangle $[-1,1]\times [-j+2,2]$.  Let $E''$ be the edge with $(0,-1)\in\cone{E''}$ that contains the lower right vertex $(1,2-j)\in E''$. If $(0,-1)\notin P$ then one immediately finds that $j=5$ and $P$ is a triangle. This is the triangle $\TpolyC$ in~Table~\ref{tab:triangle_smoothable}. If $(0,-1)\in\bdry P$ then $j\le 6$, and $P$ is either the triangle $\TpolyD$, or one of $\PpolyE$ or $\PpolyF$ in Table~\ref{tab:polygon_smoothable}. Finally, if $(0,-1)\in\intr{P}$ then the height of $E''$ must be two. Therefore $(0,-2)\in E''$, giving either the triangle $\TpolyE$, or one of $\PpolyG$, $\PpolyH$, or $\PpolyI$.

In the cases $(2,5,2)$, $(8,4,1)$, and $(9,3,1)$, we have equality in \eqref{eq:triangle}, hence $1=\Vol{\dual{P}}=\Vol{\dual{T}}$ and so $P=T$ is uniquely determined. This gives the triangles $\TpolyH$, $\TpolyG$, and $\TpolyF$. In the case $(6,3,1)$ the triangle $T=\TpolyI$. In the remaining cases it is easily verified that $T$ is not a minimal triangle with only $T$-singularities; this completes the proof of~\eqref{item:T-sing_minimals_triangle}.

Consider the cases $(6,3,1)$, $(7,3,1)$, and $(8,3,1)$. In all three cases $\rE=3$ and $b=1$, hence $P$ is contained in the rectangle $[-2,1]\times[-j+3,3]$, $j\in\{6,7,8\}$. If we assume that $P$ is not a triangle, it follows that that $(0,-1)\in\intr{P}$, therefore $(0,-1)\in\cone{E''}$ for an edge $E''\in\F{P}$ containing the bottom-right vertex $(1,3-j)$. This implies that the edge $E''$ has height two or three. Let us assume the height of $E''$ is three. Then it must have width three, and so $P$ is not a triangle there must exist one more vertex with first coordinate $-2$. However, this means that there exists a left vertical edge, contradicting minimality. Now assume that the height of $E''$ is two. Hence $(0,-2)$ is an interior point in $E''$, and since $E''$ has width two there is a unique vertex with first coordinate $-1$. Hence there can be at most one vertex left with first coordinate $-2$, excluding the vertex $(-2,3)$. Enumerating these possibilities shows that none results in a minimal Fano polygon with only $T$-cones.

Finally, consider the case $(9,4,1)$. We see that $(0,-2)\in\intr{P}$, hence the non-vertical edge $E''\in\F{P}$ containing $(1,-5)$ is of height either three or four. If $r_{E''}=3$ then $(-2,1)$ is a vertex of $P$, and if $r_{E''}=4$ then $(-3,-1)$ is a vertex of $P$. In either case $P$ fails to have only $T$-cones, and so~\eqref{item:T-sing_minimals_exceptional} is complete.
\end{proof}

\begin{cor}\label{cor:T-sing_mutation_classes_deg_8}
There are precisely two mutation-equivalence classes of Fano polygons with singularity content $(4,\emptyset)$. These classes are given by $\RpolyJ$ and $\RpolyK$, corresponding to the toric varieties $\FF_1$ and $\Proj^1\times\Proj^1$, respectively.
\end{cor}

\begin{proof}
By Theorem~\ref{thm:T-sing_minimals} there are only three minimal polygons with singularity content $(4,\emptyset)$: $\RpolyO$, $\RpolyJ$, and $\RpolyK$. It is easy to see that $\RpolyO$ and $\RpolyK$ are mutation-equivalent. That $\RpolyJ$ and $\RpolyK$ are distinct up to mutation is shown in Example~\ref{eg:smooth_case}.
\end{proof}

\begin{thm}\label{thm:T-sing_mutation_classes}
There are ten mutation-equivalence classes of Fano polygon with residual basket $\cB=\emptyset$.
Representative polygons for each mutation-equivalence class are given by \hyperlink{fig:RpolyP}{$\RpolyP$}, \hyperlink{fig:RpolyK}{$\RpolyK$}, \hyperlink{fig:RpolyJ}{$\RpolyJ$}, \hyperlink{fig:RpolyD}{$\RpolyD$}, \hyperlink{fig:RpolyA}{$\RpolyA$}, \hyperlink{fig:RpolyB}{$\RpolyB$}, \hyperlink{fig:RpolyF}{$\RpolyF$}, \hyperlink{fig:RpolyL}{$\RpolyL$}, \hyperlink{fig:PpolyJ}{$\PpolyJ$}, and \hyperlink{fig:TpolyH}{$\TpolyH$}.
\end{thm}

\begin{proof}
Let $P,Q\subset\NQ$ be two minimal Fano polygons as given in Theorem~\ref{thm:T-sing_minimals} with $\SC{P}=\SC{Q}$. With the exception of the case when $\SC{P}=\SC{Q}=(4,\emptyset)$, which is handled in Corollary~\ref{cor:T-sing_mutation_classes_deg_8} above, it can easily be seen that $P$ and $Q$ are mutation-equivalent. We will do the case when $\SC{P}=\SC{Q}=(6,\emptyset)$; the remaining cases are similar. The minimal Fano polygons $\RpolyN$, $\RpolyH$, $\RpolyC$, and $\RpolyA$ are connected via a sequence of mutations:
\begin{center}
\includegraphics[scale=0.6]{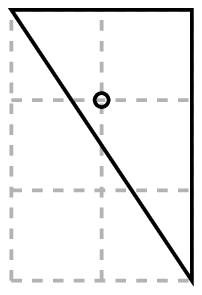}
\raisebox{32pt}{$\xmapsto{(-1,0)}$}
\raisebox{15.5pt}{\includegraphics[scale=0.6]{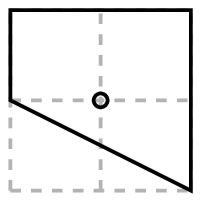}}
\raisebox{32pt}{$\xmapsto{(0,-1)}$}
\raisebox{53pt}{\includegraphics[scale=0.6,angle=-90]{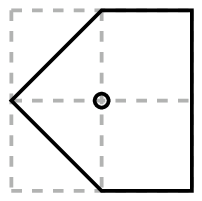}}
\raisebox{32pt}{$\xmapsto{(-1,0)}$}
\raisebox{15.5pt}{\includegraphics[scale=0.6]{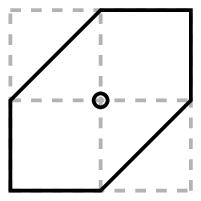}}
\end{center}
The mutations have been labelled with their corresponding primitive inner normal vector $w$.
\end{proof}

%-------------------------------------------------------------------------------
\section{Finiteness of minimal Fano polygons}\label{sec:general_case}
%-------------------------------------------------------------------------------
In this section we generalise Theorem~\ref{thm:T-sing_minimals} to the case when the residual basket $\cB\neq\emptyset$.

\begin{definition}\label{defn:basket_numerics}
Given a residual basket $\cB\neq\emptyset$ we define
\begin{align*}
\mB&:=\max{r_\sigma\mid\sigma\in\cB},\\
\dB&:=\lcm{\denom{A_\sigma}\mid\sigma\in\cB},\\
\sB&:=-\bmin{\{0\}\cup\{A_\sigma\mid\sigma\in\cB\}},
\end{align*}
where $r_\sigma$ is the Gorenstein index of $\sigma$, $A_\sigma$ is the contribution of $\sigma$ to the degree (as given in Proposition~\ref{prop:sing_content_degree}), and $\denom{x}$ denotes the denominator of $x\in\Q$. In the case when $\cB=\emptyset$ we define $\mB:=1$, $\dB:=1$, and $\sB:=0$.
\end{definition}

\begin{remark}
Bounding $\mB$ automatically bounds the number of possible types of singularities that can occur in the residual basket. In particular there are only finitely many possible values of $A_\sigma$, hence $\dB$ and $\sB$ are bounded from above.
\end{remark}

\begin{thm}\label{thm:general_minimals}
There exist only a finite number of minimal Fano polygons, up to the action of $\GL_2(\Z)$, with bounded maximal local index $\mB$ of the cones in the residual baskets $\cB$.
\end{thm}

\begin{proof}
We assume throughout that $\cB\neq\emptyset$, the empty case having already been considered in Theorem~\ref{thm:T-sing_minimals}. The proof is constructive, and follows a similar structure to the proof when $\cB=\emptyset$.

\proofsection{$\mP=\mB$}
The number of possible Fano polygons with bounded maximal local index $\mP$ is known to be finite, up to $\GL_2(\Z)$-equivalence. Algorithms for computing all such Fano polygons are described in~\cite{KKN08}.

\proofsection{$\mP>\mB$}
Let $E\in\F{P}$ be an edge of maximum local index $\rE=\mP>\mB$. In particular $\cone{E}$ is a $T$-cone, hence $\abs{E\cap N}=k\rE+1$ for some $k\in\Z_{>0}$. After suitable change of basis we can insist that $E=\sconv{(-a,\rE),(b,\rE)}$ where $a,b>0$, $a<\rE$, and $a+b=k\rE$. Hence, as in the proof of Theorem~\ref{thm:T-sing_minimals}, there exist $k$ distinct points $v_i=(i-1,1)$, for $i=1,\ldots,k$, where each $v_i\in\intr{\cone{E}}$, $w_E(v_i)=-1$. Let $v\in\V{P}$ be a vertex of $P$ such that $w_E(v)=\max{w_E(v')\mid v'\in P}$. Since $P$ is minimal by assumption, by Corollary~\ref{cor:minimise_heights} we have that $w_E(v)\geq\rE$. By applying Lemma~\ref{lem:edge_projection} with respect to $v_1$ and $v_k$ we conclude that $k\leq 2$.

\proofsection{$k=2$}
First we consider the case when $k=2$. As in the proof of Theorem~\ref{thm:T-sing_minimals} we have that
$$P=\sconv{(-a,\rE),(-a+2\rE,\rE),(-a,-\rE)},\quad\text{ where }0<a<\rE, \gcd{a,\rE}=1.$$
Let $E_1$ and $E_2$ be the two edges of $P$ distinct from the horizontal edge $E$, where $r_{E_1}=a$ and $r_{E_2}=\rE-a$. Since at least one of these two edges has local index $\mB$, by symmetry we can assume that $a=\mB$. The edge $E_2$ is of width $2\rE$, giving $2\rE=j(\rE-\mB)+l$ for some $j\in\Z_{\geq 0}$, $0\leq l<\rE-\mB$.

If $l=0$ then $j\mB=(j-2)\rE$, and $\gcd{\mB,\rE}=1$ implies that $\rE\divides j$. Writing $j=j'\rE$ for some $j'\in\Z_{\geq 0}$ we see that $2=j'(\rE-\mB)$ and hence $\rE=\mB+1$ or $\rE=\mB+2$. If $l>0$ then $\rE-\mB\leq\mB$, and so $\rE<2\mB$ (the case of equality being excluded by primitivity). Hence in either case the number of possible minimal polygons is finite.

\proofsection{$k=1$}
We now consider the case when $k=1$. Once again we subdivide this into two cases depending on whether there exists an edge $E'\in\F{P}$ parallel to $E$.

\proofsection{Edge $E'$ parallel to $E$}
First we assume that there exists an edge $E'\in\F{P}$ such that $v$ is a vertex of $E'$, and $E'$ is parallel to $E$. If $\cone{E'}$ contains a residual component then $r_{E'}\leq\mB$. By minimality $\rE\le\mB$, which is a contradiction. Therefore $\cone{E'}$ must be of class $T$, and by minimality we see that $r_{E'}=\rE$. As in the proof of Theorem~\ref{thm:T-sing_minimals} we conclude that $P$ is a rectangle:
$$P=\sconv{(-a,\rE),(-a+\rE,\rE),(-a,-\rE),(-a+\rE,-\rE)}.$$
Since one of the two vertical edges must contain a residual component at height $\mB$, without loss of generality we may assume that $a=\mB$, hence $2\rE=j\mB+l$ for some $j\in\Z_{\geq 0}$, $0<l<\mB$. Notice that if $j=0$ then $2\rE=l<\mB$, a contradiction. Hence $j>0$. The second vertical edge lies at height $\rE-\mB$, and by Corollary~\ref{cor:minimise_heights} we have that $2\mB\leq\rE$. Hence $2\rE=j'(\rE-\mB)+l'$ for some $j'\in\Z_{\geq 0}$, $0\leq l'<\rE-\mB$. If $j'=0$ then $2\rE=l'<\rE-\mB$, a contradiction. If $j>0$ then, by minimality, $\rE\leq 2\mB$, implying that $\rE=2\mB$. But this contradicts primitivity, hence this case does not occur.

\proofsection{No edge parallel to $E$}
We now assume that $v$ is the unique point in $P$ such that $w_E(v)=\mathrm{max}\{w_E(v')\mid v'\in P\}$. Once again we subdivide this into two cases, depending on whether there exists a vertical edge $E'\in\F{P}$ with vertex $v$.

\proofsection{$E'$ not vertical}
This proof in this case is identical to that of Theorem~\ref{thm:T-sing_minimals}: it results in no minimal polygons.

\proofsection{$E'$ vertical}
Without loss of generality we may assume that there is a vertical edge $E'\in\F{P}$ with vertices $v$ and $(b,\rE)$. Hence $v=(b,\rE-jb-l)$ for some $j\in\Z_{\geq 0}$, $0\leq l<b$. By minimality of $E$,
\begin{equation}\label{ineq:b-min_general}
2\rE\le jb+l.
\end{equation}
Notice that if $j=0$ then $2\rE\leq l<\mB$, a contradiction. Hence $j>0$ and, by minimality of $E'$,
\begin{equation}\label{ineq:b-max_general}
2b\le\rE.
\end{equation}

\proofsection{$l=0$}
Consider the case when $l=0$. Inequalities~\eqref{ineq:b-min_general} and~\eqref{ineq:b-max_general} imply that $j\ge 4$.

\proofsection{$j=4$}
Assume that $j=4$. Then $2b=\rE$, so $\gcd{b,\rE}=1$ implies that $\rE=2$ and $b=1$. Hence $P$ is contained in the rectangle $[-1,1]\times [-2,2]$. This contains only four possible polygons, all of which have $\mP\leq 3$, contradicting $\mP>\mB\geq 3$. Hence this case does not occur.

\proofsection{$j\geq 5$}
As in the proof of  Theorem~\ref{thm:T-sing_minimals}, $\orig$ is contained in the strict interior of the triangle
$$T:=\sconv{(b-\rE,\rE),(b,\rE),(b,\rE-jb)}\subset P.$$
As before, the volume of the dual triangle is given by
$$\Vol{\dual{T}}=\frac{j}{jb\rE-jb^2-\rE^2}.$$
By Proposition~\ref{prop:sing_content_degree}, $\Vol{\dual{P}}\in\frac{1}{\dB}\cdot\Z_{>0}$, hence $\frac{1}{\dB}\le\Vol{\dual{P}}\le\Vol{\dual{T}}$. Let $b=\rE q$, where $\frac{2}{j}\leq q\leq\frac{1}{2}$, by~\eqref{ineq:b-min_general} and~\eqref{ineq:b-max_general}. Then $\rE^2(-jq^2+jq-1)\le j\dB$, and by considering the minimum value achieved by the quadratic in $q$ on the left-hand-side of this inequality we obtain
\begin{equation}\label{eq:triangle_general}
\rE^2\le\frac{j^2\dB}{j-4}=\left(j+4+\frac{16}{j-4}\right)\dB\leq(j+20)\dB.
\end{equation}

Notice that the number of edges of $P$ distinct from $E$ and $E'$ can be at most $\rE+1$. By Proposition~\ref{prop:sing_content_degree} we obtain that $12-1-j+(\rE+1)\sB\ge\Vol{\dual{P}}>0$. Hence
\begin{equation}\label{eq:j-bound_general}
j<11+(\rE+1)\sB.
\end{equation}
Combining inequalities~\eqref{eq:triangle_general} and~\eqref{eq:j-bound_general} gives $\rE^2<(31+(\rE+1)\sB)\dB$, which implies that
\begin{equation}\label{eq:l-bound_r=0_general}
\rE<\frac{\sB\dB+\sqrt{(\sB\dB)^2+4\dB(\sB+31)}}{2}.
\end{equation}
Since $j$ is bounded by~\eqref{eq:j-bound_general}, the result follows.

\proofsection{$l>0$}
Finally, let us consider the case when the edge $E'$ contributes a residual singularity to the basket $\cB$. Notice that if we had equality in~\eqref{ineq:b-max_general}, primitivity forces $\rE=2$ and $b=1$, and so $l=0$, which is a contradiction. Hence we conclude that the inequality is strict:
\begin{equation}\label{ineq:b-max_residual_general}
2b<\rE.
\end{equation}
Once again we find that $j\geq 4$.

Consider the triangle
$$T:=\sconv{(b-\rE,\rE),(b,\rE),(b,\rE-jb-l)}\subset P.$$
One can check that $\orig\in\intr{T}$, so $T$ is a Fano triangle with $\Vol{T}=\rE(jb+l)$. Moreover the pairwise coprime weights
$$\frac{1}{\gcd{jb+l,\rE^2}}\left(b(jb+l),(jb+l)(\rE-b)-\rE^2,\rE^2\right)\in\Z_{>0}^3$$
satisfy the conditions of Lemma~\ref{lem:triangle_degree} (that the second weight is strictly positive follows from~\eqref{ineq:b-min_general} and~\eqref{ineq:b-max_residual_general}: $(jb+l)(\rE-b)-\rE^2\geq 2\rE^2-2\rE b-\rE^2=\rE(\rE-2b)>0$). Hence
$$\Vol{\dual{T}}=\frac{jb+l}{b(jb+l)(\rE-b)-b\rE^2}<\frac{j+1}{(jb+l)(\rE-b)-\rE^2}.$$
Recalling, as above, that $\frac{1}{\dB}\le\Vol{\dual{P}}\leq\Vol{\dual{T}}$ we obtain:
\begin{equation}\label{eq:triangle_residual_general}
(jb+l)(\rE-b)-\rE^2<(j+1)\dB.
\end{equation}
The quadratic in $b$ on the left-hand-side of~\eqref{eq:triangle_residual_general} is strictly positive in the range $\frac{2\rE-l}{j}\leq b<\frac{\rE}{2}$, and obtains its minimum value when $b=\frac{2\rE-l}{j}$. Hence~\eqref{eq:triangle_residual_general} gives:
\begin{equation}\label{eq:residual_general_bound}
(j-4)\rE^2+2l\rE<j(j+1)\dB.
\end{equation}
We consider the cases $j=4$ and $j>4$ separately.

\proofsection{$j=4$}
When $j=4$ inequalities~\eqref{ineq:b-min_general} and~\eqref{ineq:b-max_residual_general} give $\rE-\frac{l}{2}\leq 2b<\rE$. Hence if $l=1$ this case does not occur. If $l>1$ then~\eqref{eq:residual_general_bound} gives us that $\rE<10\dB$, resulting in only finitely many possibilities.

\proofsection{$j\geq 5$}
When $j\geq 5$ inequality~\eqref{eq:residual_general_bound} implies that
\begin{equation}\label{eq:residual_general_bound2}
\rE^2<\frac{j(j+1)\dB}{j-4}=\left(j+5+\frac{20}{j-4}\right)\dB\leq(j+25)\dB.
\end{equation}
Notice that the number of edges of $P$ distinct from $E$ and $E'$ is at most $\rE+1$, so by Proposition~\ref{prop:sing_content_degree} we have that $12-1-j+(\rE+2)\sB\ge\Vol{\dual{P}}>0$. Hence
\begin{equation}\label{eq:j-bound_general2}
j<11+(\rE+2)\sB.
\end{equation}
Combining inequalities~\eqref{eq:residual_general_bound2} and~\eqref{eq:j-bound_general2} gives $\rE^2<(36+(\rE+2)\sB)\dB$. This implies that
\begin{equation}\label{eq:l-bound_r>0_general}
\rE<\frac{\sB\dB+\sqrt{(\sB\dB)^2+4\dB(2\sB+36)}}{2}.
\end{equation}
Since $j$ is bounded by~\eqref{eq:j-bound_general2} we have only finitely many possible minimal polygons.
\end{proof}

%-------------------------------------------------------------------------------
\section{Minimal Fano polygons with $\cB=\{m\times\third\}$}\label{sec:third_one_one}
%-------------------------------------------------------------------------------
In this section we apply Theorem~\ref{thm:general_minimals} in order to classify all minimal Fano polygons with residual basket $\cB$ containing only singularities of type $\third$. We find $64$ minimal Fano polygons (Theorem~\ref{thm:minimal_third_one_one} and Table~\ref{tab:minimal_third_one_one}), which result in $26$ mutation-equivalence classes (Theorem~\ref{thm:classes_third_one_one} and Table~\ref{tab:third_one_one_polygons}). These mutation-equivalence classes correspond exactly with the classification of Corti--Heuberger of qG-deformation-equivalence classes of del~Pezzo surfaces of class TG with $m\times\third$ singular points~\cite{CH}.

In some sense $\third$ is the ``simplest'' residual singularity. Up to change of basis, the corresponding cone is given by $C:=\scone{(1,0),(2,3)}$. The width of the line segment joining the primitive generators of the rays of $C$ is one. The local index is three. By Example~\ref{eg:degree_hilb_third_one_one} any Fano polygon $P$ with singularity content $\left(n,\{m\times\third\}\right)$ gives rise to a toric surface $X$ with degree
$$(-K_X)^2=12-n-\frac{5m}{3}.$$
In the notation of Definition~\ref{defn:basket_numerics} $\mB=\dB=3$ and $\sB=0$.

\begin{thm}\label{thm:minimal_third_one_one}
Let $P\subset\NQ$ be a minimal Fano polygon with residual basket $\cB=\{m\times\third\}$, for some $m\geq 1$. Then, up to $\GL_2(\Z)$-equivalence, $P$ is one of the $64$ polygons listed in Table~\ref{tab:minimal_third_one_one}.
\end{thm}

\begin{proof}
To derive the classification, we follow the proof of Theorem~\ref{thm:general_minimals}.

\proofsection{$\mP=3$}
Techniques for classifying all Fano polygons with given maximum local index $\mP$ (or with given Gorenstein index $r$) are described in~\cite{KKN08}. The resulting classifications for low index are available online~\cite{GRDb}, and it is a simple process to sift these results for the minimal polygons we require. There are precisely $60$ such polygons. These are the polygons in Table~\ref{tab:minimal_third_one_one}, excluding numbers~\hyperlink{row:1.1}{$(1.1)$}, \hyperlink{row:1.2}{$(1.2)$}, \hyperlink{row:1.3}{$(1.3)$}, and~\hyperlink{row:2.6}{$(2.6)$}.

\proofsection{$\mP>3, k=2$}
In this case $P$ is a triangle with vertices $(-3,\rE)$, $(-3+2\rE,\rE)$, and $(-3,-\rE)$, where $\rE=4$ or $5$. We require that $2\rE\equiv 1\ \modb{3}$, which excludes $\rE=4$. This gives polygon number~\hyperlink{row:1.1}{$(4.4)$} in Table~\ref{tab:minimal_third_one_one}.

\proofsection{$\mP>3$, $k=1$, no edge parallel to $E$, $E'$ vertical, $l=0$, $j\geq 5$}
In this case $P$ has vertices $(b-\rE,\rE)$, $(b,\rE)$, $(b,\rE-jb)$, and is contained in the rectangle $[b-\rE,b]\times [\rE-jb,\rE]$ where $5\leq j<11$, by~\eqref{eq:j-bound_general}, $3<\rE\leq 9$, by~\eqref{eq:l-bound_r=0_general}, and $\frac{2\rE}{j}\leq b\leq\frac{\rE}{2}$ by~\eqref{ineq:b-min_general} and~\eqref{ineq:b-max_general}. This gives three minimal polygons: numbers~\hyperlink{row:1.2}{$(1.2)$}, \hyperlink{row:1.3}{$(1.3)$}, and~\hyperlink{row:2.6}{$(2.6)$} in Table~\ref{tab:minimal_third_one_one}.

\proofsection{$\mP>3$, $k=1$, no edge parallel to $E$, $E'$ vertical, $l>0$, $j\geq 5$}
In this case $P$ has vertices $(b-\rE,\rE)$, $(b,\rE)$, $(b,\rE-jb-1)$, and is contained in the rectangle $[b-\rE,b]\times [\rE-jb-1,\rE]$ where $5\leq j<11$, by~\eqref{eq:j-bound_general2}, $3<\rE\leq 10$, by~\eqref{eq:l-bound_r>0_general}, and $\frac{2\rE}{j}\leq b<\frac{\rE}{2}$, by~\eqref{ineq:b-min_general} and~\eqref{ineq:b-max_residual_general}. This results in no minimal polygons.
\end{proof}
%-------------------------------------------------------------------------------
\begin{center}
\begin{longtable}{rlccccc}
\caption{The $64$ minimal Fano polygons $P\subset\NQ$ with singularity content $\left(n,\{m\times\third\}\right)$, $m\geq 1$, and vertices $\V{P}$, up to the action of $\GL_2(\Z)$, with Gorenstein index $r$ and maximum local index $\mP$. The degrees $(-K_X)^2=12-n-\frac{5m}{3}$ of the corresponding toric varieties are also given. The polygons are partitioned into $26$ mutation equivalence classes; the invariants $n$ and $m$ completely determine the mutation equivalence class except when $n=6$, $m=2$, $(-K_X)^2=\frac{8}{3}$, and when $n=5$, $m=1$, $(-K_X)^2=\frac{16}{3}$. See Theorems~\ref{thm:minimal_third_one_one} and~\ref{thm:classes_third_one_one}.}
\label{tab:minimal_third_one_one}\\
\toprule
\multicolumn{1}{c}{\#}&\multicolumn{1}{c}{$\V{P}$}&$r$&$\mP$&$n$&$m$&$(-K_X)^2$\\
\cmidrule(lr){1-1} \cmidrule(lr){2-2} \cmidrule(lr){3-4} \cmidrule(lr){5-7}
\endfirsthead
\multicolumn{7}{l}{\vspace{-0.7em}\tiny Continued from previous page.}\\
\addlinespace[1.7ex]
\midrule
\multicolumn{1}{c}{\#}&\multicolumn{1}{c}{$\V{P}$}&$r$&$\mP$&$n$&$m$&$(-K_X)^2$\\
\cmidrule(lr){1-1} \cmidrule(lr){2-2} \cmidrule(lr){3-4} \cmidrule(lr){5-7}
\endhead
\multicolumn{7}{r}{\raisebox{0.2em}{\tiny Continued on next page.}}\\
\endfoot
\bottomrule
\endlastfoot
% Bucket: 1  Degree: 1/3  n: 10  m: 1
\oddrow \padding $1.1$&\gap $(7,5),(-3,5),(-3,-5)$&$30$&$5$&$10$&$1$&$\frac{1}{3}$\\
\oddrow \padding $1.2$&\gap $(2,5),(-3,5),(-3,4),(2,-11)$&$30$&$5$&$10$&$1$&$\frac{1}{3}$\\
\oddrow \padding $1.3$&\gap $(2,7),(-5,7),(2,-7)$&$42$&$7$&$10$&$1$&$\frac{1}{3}$\\
\cmidrule(lr){1-1} \cmidrule(lr){2-2} \cmidrule(lr){3-4} \cmidrule(lr){5-7}
% Bucket: 2  Degree: 2/3  n: 8  m: 2
\evnrow \padding $2.1$&\gap $(11,2),(-1,2),(-5,-2)$&$6$&$3$&$8$&$2$&$\frac{2}{3}$\\
\evnrow \padding $2.2$&\gap $(5,2),(-1,2),(-2,1),(-2,-5)$&$6$&$3$&$8$&$2$&$\frac{2}{3}$\\
\evnrow \padding $2.3$&\gap $(5,2),(-1,2),(-5,-2),(1,-2)$&$6$&$3$&$8$&$2$&$\frac{2}{3}$\\
\evnrow \padding $2.4$&\gap $(3,2),(-1,2),(-5,-2),(3,-2)$&$6$&$3$&$8$&$2$&$\frac{2}{3}$\\
\evnrow \padding $2.5$&\gap $(1,2),(-1,2),(-5,-2),(5,-2)$&$6$&$3$&$8$&$2$&$\frac{2}{3}$\\
\evnrow \padding $2.6$&\gap $(2,5),(-3,5),(-3,4),(1,-4),(2,-5)$&$30$&$5$&$8$&$2$&$\frac{2}{3}$\\
\cmidrule(lr){1-1} \cmidrule(lr){2-2} \cmidrule(lr){3-4} \cmidrule(lr){5-7}
% Bucket: 3  Degree: 1  n: 6  m: 3
\oddrow \padding $3.1$&\gap $(7,2),(-1,2),(-2,1),(-2,-1),(-1,-2)$&$6$&$3$&$6$&$3$&$1$\\
\oddrow \padding $3.2$&\gap $(3,1),(3,2),(-1,2),(-2,1),(-2,-3),(-1,-3)$&$6$&$3$&$6$&$3$&$1$\\
\oddrow \padding $3.3$&\gap $(2,1),(1,2),(-1,2),(-2,1),(-2,-1),(2,-5)$&$6$&$3$&$6$&$3$&$1$\\
\cmidrule(lr){1-1} \cmidrule(lr){2-2} \cmidrule(lr){3-4} \cmidrule(lr){5-7}
% Bucket: 4  Degree: 4/3  n: 9  m: 1
\evnrow \padding \hypertarget{row:4.1}{}$4.1$&\gap $(7,2),(-1,2),(-5,-2)$&$6$&$3$&$9$&$1$&$\frac{4}{3}$\\
\evnrow \padding \hypertarget{row:4.2}{}$4.2$&\gap $(3,2),(-1,2),(-2,1),(-2,-3)$&$6$&$3$&$9$&$1$&$\frac{4}{3}$\\
\evnrow \padding \hypertarget{row:4.3}{}$4.3$&\gap $(1,2),(-1,2),(-2,1),(1,-5)$&$6$&$3$&$9$&$1$&$\frac{4}{3}$\\
\evnrow \padding \hypertarget{row:4.4}{}$4.4$&\gap $(1,2),(-1,2),(-5,-2),(1,-2)$&$6$&$3$&$9$&$1$&$\frac{4}{3}$\\
\cmidrule(lr){1-1} \cmidrule(lr){2-2} \cmidrule(lr){3-4} \cmidrule(lr){5-7}
% Bucket: 5  Degree: 4/3  n: 4  m: 4
\oddrow \padding $5.1$&\gap $(5,2),(-1,2),(-2,1),(-1,-1),(1,-2)$&$6$&$3$&$4$&$4$&$\frac{4}{3}$\\
\oddrow \padding $5.2$&\gap $(3,1),(3,2),(-1,2),(-2,1),(-2,-1),(-1,-2),(1,-1)$&$6$&$3$&$4$&$4$&$\frac{4}{3}$\\
\oddrow \padding $5.3$&\gap $(2,1),(1,2),(-1,2),(-2,1),(-2,-1),(-1,-2),(1,-2),(2,-1)$&$6$&$3$&$4$&$4$&$\frac{4}{3}$\\
\cmidrule(lr){1-1} \cmidrule(lr){2-2} \cmidrule(lr){3-4} \cmidrule(lr){5-7}
% Bucket: 6  Degree: 5/3  n: 7  m: 2
\evnrow \padding $6.1$&\gap $(1,0),(2,3),(-1,3),(-1,-3)$&$3$&$3$&$7$&$2$&$\frac{5}{3}$\\
\evnrow \padding $6.2$&\gap $(5,2),(-1,2),(-2,1),(1,-2)$&$6$&$3$&$7$&$2$&$\frac{5}{3}$\\
\evnrow \padding $6.3$&\gap $(5,2),(-1,2),(-5,-2),(-1,-1)$&$6$&$3$&$7$&$2$&$\frac{5}{3}$\\
\evnrow \padding $6.4$&\gap $(2,1),(1,2),(-1,2),(-2,1),(-2,-3)$&$6$&$3$&$7$&$2$&$\frac{5}{3}$\\
\evnrow \padding $6.5$&\gap $(1,2),(-1,2),(-2,1),(-1,-1),(1,-3)$&$6$&$3$&$7$&$2$&$\frac{5}{3}$\\
\cmidrule(lr){1-1} \cmidrule(lr){2-2} \cmidrule(lr){3-4} \cmidrule(lr){5-7}
% Bucket: 7  Degree: 5/3  n: 2  m: 5
\oddrow \padding $7.1$&\gap $(2,1),(1,2),(-1,2),(-2,1),(-2,-1),(-1,-2),(1,-1)$&$6$&$3$&$2$&$5$&$\frac{5}{3}$\\
\cmidrule(lr){1-1} \cmidrule(lr){2-2} \cmidrule(lr){3-4} \cmidrule(lr){5-7}
% Bucket: 8  Degree: 2  n: 5  m: 3
\evnrow \padding $8.1$&\gap $(1,2),(-1,2),(-2,1),(-1,-1),(1,-2)$&$6$&$3$&$5$&$3$&$2$\\
\evnrow \padding $8.2$&\gap $(2,1),(1,2),(-1,2),(-2,1),(-2,-1),(-1,-2)$&$6$&$3$&$5$&$3$&$2$\\
\cmidrule(lr){1-1} \cmidrule(lr){2-2} \cmidrule(lr){3-4} \cmidrule(lr){5-7}
% Bucket: 9  Degree: 2  n: 0  m: 6
\oddrow \padding $9.1$&\gap $(1,0),(-1,3),(-2,3),(-1,0),(1,-3),(2,-3)$&$3$&$3$&$0$&$6$&$2$\\
\cmidrule(lr){1-1} \cmidrule(lr){2-2} \cmidrule(lr){3-4} \cmidrule(lr){5-7}
% Bucket: 10  Degree: 7/3  n: 8  m: 1
\evnrow \padding $10.1$&\gap $(7,1),(0,1),(-4,-1)$&$6$&$3$&$8$&$1$&$\frac{7}{3}$\\
\evnrow \padding $10.2$&\gap $(5,2),(-1,2),(-1,0),(1,-2)$&$6$&$3$&$8$&$1$&$\frac{7}{3}$\\
\evnrow \padding $10.3$&\gap $(1,1),(0,1),(-4,-1),(2,-1)$&$6$&$3$&$8$&$1$&$\frac{7}{3}$\\
\evnrow \padding $10.4$&\gap $(2,1),(0,1),(-4,-1),(1,-1)$&$6$&$3$&$8$&$1$&$\frac{7}{3}$\\
\evnrow \padding $10.5$&\gap $(3,1),(0,1),(-4,-1),(0,-1)$&$6$&$3$&$8$&$1$&$\frac{7}{3}$\\
\evnrow \padding $10.6$&\gap $(1,2),(-1,2),(-2,1),(1,-2)$&$6$&$3$&$8$&$1$&$\frac{7}{3}$\\
\cmidrule(lr){1-1} \cmidrule(lr){2-2} \cmidrule(lr){3-4} \cmidrule(lr){5-7}
% Bucket: 11  Degree: 7/3  n: 3  m: 4
\oddrow \padding $11.1$&\gap $(1,0),(-1,3),(-2,3),(-1,0),(1,-3)$&$3$&$3$&$3$&$4$&$\frac{7}{3}$\\
\cmidrule(lr){1-1} \cmidrule(lr){2-2} \cmidrule(lr){3-4} \cmidrule(lr){5-7}
% Bucket: 12  Degree: 8/3  n: 6  m: 2
\evnrow \padding $12.1$&\gap $(6,1),(0,1),(-3,-1)$&$3$&$3$&$6$&$2$&$\frac{8}{3}$\\
\evnrow \padding $12.2$&\gap $(2,1),(0,1),(-3,-1),(1,-1)$&$3$&$3$&$6$&$2$&$\frac{8}{3}$\\
\evnrow \padding $12.3$&\gap $(3,1),(0,1),(-3,-1),(0,-1)$&$3$&$3$&$6$&$2$&$\frac{8}{3}$\\
\evnrow \padding $12.4$&\gap $(1,1),(0,1),(-3,-1),(2,-1)$&$3$&$3$&$6$&$2$&$\frac{8}{3}$\\
\cmidrule(lr){1-1} \cmidrule(lr){2-2} \cmidrule(lr){3-4} \cmidrule(lr){5-7}
% Bucket: 13  Degree: 8/3  n: 6  m: 2
\oddrow \padding $13.1$&\gap $(1,0),(-1,3),(-2,3),(1,-3)$&$3$&$3$&$6$&$2$&$\frac{8}{3}$\\
\cmidrule(lr){1-1} \cmidrule(lr){2-2} \cmidrule(lr){3-4} \cmidrule(lr){5-7}
% Bucket: 14  Degree: 3  n: 4  m: 3
\evnrow \padding $14.1$&\gap $(-1,3),(-2,3),(-1,0),(1,-3),(1,-1)$&$3$&$3$&$4$&$3$&$3$\\
\cmidrule(lr){1-1} \cmidrule(lr){2-2} \cmidrule(lr){3-4} \cmidrule(lr){5-7}
% Bucket: 15  Degree: 10/3  n: 7  m: 1
\oddrow \padding $15.1$&\gap $(5,1),(0,1),(-3,-1)$&$3$&$3$&$7$&$1$&$\frac{10}{3}$\\
\oddrow \padding $15.2$&\gap $(1,1),(0,1),(-3,-1),(1,-1)$&$3$&$3$&$7$&$1$&$\frac{10}{3}$\\
\oddrow \padding $15.3$&\gap $(2,1),(0,1),(-3,-1),(0,-1)$&$3$&$3$&$7$&$1$&$\frac{10}{3}$\\
\oddrow \padding $15.4$&\gap $(5,2),(-1,2),(-1,1),(0,-1),(1,-2)$&$6$&$3$&$7$&$1$&$\frac{10}{3}$\\
\cmidrule(lr){1-1} \cmidrule(lr){2-2} \cmidrule(lr){3-4} \cmidrule(lr){5-7}
% Bucket: 16  Degree: 11/3  n: 5  m: 2
\evnrow \padding $16.1$&\gap $(-1,3),(-2,3),(-1,0),(1,-2),(1,-1)$&$3$&$3$&$5$&$2$&$\frac{11}{3}$\\
\cmidrule(lr){1-1} \cmidrule(lr){2-2} \cmidrule(lr){3-4} \cmidrule(lr){5-7}
% Bucket: 17  Degree: 4  n: 3  m: 3
\oddrow \padding $17.1$&\gap $(0,1),(-1,3),(-2,3),(-1,0),(1,-3),(1,-2)$&$3$&$3$&$3$&$3$&$4$\\
\cmidrule(lr){1-1} \cmidrule(lr){2-2} \cmidrule(lr){3-4} \cmidrule(lr){5-7}
% Bucket: 18  Degree: 13/3  n: 6  m: 1
\evnrow \padding $18.1$&\gap $(4,1),(0,1),(-1,0),(-1,-1)$&$3$&$3$&$6$&$1$&$\frac{13}{3}$\\
\evnrow \padding $18.2$&\gap $(1,0),(1,1),(0,1),(-3,-1),(0,-1)$&$3$&$3$&$6$&$1$&$\frac{13}{3}$\\
\evnrow \padding $18.3$&\gap $(2,1),(0,1),(-1,0),(-1,-1),(1,-1)$&$3$&$3$&$6$&$1$&$\frac{13}{3}$\\
\evnrow \padding $18.4$&\gap $(5,2),(-1,2),(-5,-2),(-2,-1)$&$6$&$3$&$6$&$1$&$\frac{13}{3}$\\
\evnrow \padding $18.5$&\gap $(5,2),(-1,2),(-1,1),(1,-2)$&$6$&$3$&$6$&$1$&$\frac{13}{3}$\\
\cmidrule(lr){1-1} \cmidrule(lr){2-2} \cmidrule(lr){3-4} \cmidrule(lr){5-7}
% Bucket: 19  Degree: 14/3  n: 4  m: 2
\oddrow \padding $19.1$&\gap $(0,1),(-1,3),(-2,3),(-1,0),(1,-2)$&$3$&$3$&$4$&$2$&$\frac{14}{3}$\\
\cmidrule(lr){1-1} \cmidrule(lr){2-2} \cmidrule(lr){3-4} \cmidrule(lr){5-7}
% Bucket: 20  Degree: 5  n: 2  m: 3
\evnrow \padding $20.1$&\gap $(0,1),(-1,3),(-2,3),(-1,0),(1,-3)$&$3$&$3$&$2$&$3$&$5$\\
\cmidrule(lr){1-1} \cmidrule(lr){2-2} \cmidrule(lr){3-4} \cmidrule(lr){5-7}
% Bucket: 21  Degree: 16/3  n: 5  m: 1
\oddrow \padding $21.1$&\gap $(4,1),(0,1),(-3,-1)$&$3$&$3$&$5$&$1$&$\frac{16}{3}$\\
\oddrow \padding $21.2$&\gap $(2,1),(0,1),(-3,-1),(-1,-1)$&$3$&$3$&$5$&$1$&$\frac{16}{3}$\\
\oddrow \padding $21.3$&\gap $(1,1),(0,1),(-3,-1),(0,-1)$&$3$&$3$&$5$&$1$&$\frac{16}{3}$\\
\cmidrule(lr){1-1} \cmidrule(lr){2-2} \cmidrule(lr){3-4} \cmidrule(lr){5-7}
% Bucket: 22  Degree: 16/3  n: 5  m: 1
\evnrow \padding $22.1$&\gap $(3,1),(0,1),(-1,0),(0,-1)$&$3$&$3$&$5$&$1$&$\frac{16}{3}$\\
\evnrow \padding $22.2$&\gap $(1,0),(1,1),(0,1),(-3,-1),(-1,-1)$&$3$&$3$&$5$&$1$&$\frac{16}{3}$\\
\cmidrule(lr){1-1} \cmidrule(lr){2-2} \cmidrule(lr){3-4} \cmidrule(lr){5-7}
% Bucket: 23  Degree: 17/3  n: 3  m: 2
\oddrow \padding $23.1$&\gap $(-1,3),(-2,3),(-1,0),(1,-2)$&$3$&$3$&$3$&$2$&$\frac{17}{3}$\\
\cmidrule(lr){1-1} \cmidrule(lr){2-2} \cmidrule(lr){3-4} \cmidrule(lr){5-7}
% Bucket: 24  Degree: 19/3  n: 4  m: 1
\evnrow \padding $24.1$&\gap $(2,1),(0,1),(-1,0),(1,-1)$&$3$&$3$&$4$&$1$&$\frac{19}{3}$\\
\evnrow \padding $24.2$&\gap $(1,0),(1,1),(0,1),(-3,-1),(-2,-1)$&$3$&$3$&$4$&$1$&$\frac{19}{3}$\\
\cmidrule(lr){1-1} \cmidrule(lr){2-2} \cmidrule(lr){3-4} \cmidrule(lr){5-7}
% Bucket: 25  Degree: 22/3  n: 3  m: 1
\oddrow \padding $25.1$&\gap $(1,0),(1,1),(0,1),(-3,-1)$&$3$&$3$&$3$&$1$&$\frac{22}{3}$\\
\cmidrule(lr){1-1} \cmidrule(lr){2-2} \cmidrule(lr){3-4} \cmidrule(lr){5-7}
% Bucket: 26  Degree: 25/3  n: 2  m: 1
\evnrow \padding $26.1$&\gap $(-1,3),(-2,3),(1,-2)$&$3$&$3$&$2$&$1$&$\frac{25}{3}$\\
\end{longtable}
\end{center}
%-------------------------------------------------------------------------------

We now use the minimal polygons from Theorem~\ref{thm:minimal_third_one_one} to generate a complete list of mutation classes of Fano polygons with singularity content $\left(n,\{m\times\third\}\right)$, $m\geq 1$. Theorem~\ref{thm:classes_third_one_one}.  For future reference we recall some of the polytopes from Table~\ref{tab:minimal_third_one_one}:

\begin{lemma}\label{lem:third_one_one_6_2}
There are exactly two mutation-equivalence classes of Fano polygons with singularity content $\left(6,\{2\times\third\}\right)$. The mutation-equivalence classes are generated by
$$
P_{12}=\sconv{(6,1),(0,1),(-3,-1)}
\qquad\text{ and }\qquad 
P_{13}=\sconv{(1,0),(-1,3),(-2,3),(1,-3)}.
$$
\end{lemma}

\begin{proof}
Notice that the six primitive $T$-cones in $P_{12}$ are all contributed by the (cone over the) edge $E$ with inner normal vector $(0,-1)\in M$. Hence $\Tlattice{P_{12}}$ is a one-dimensional sublattice of $M$. The polygon $P_{13}$ has primitive $T$-cones contributed by those edges with inner normal vectors $(-1,0)$ and $(2,1)\in M$, hence $\Tlattice{P_{13}}$ equals $M$. By Lemma~\ref{lem:sublattice_index_preserved} we conclude that $P_{12}$ cannot be mutation-equivalent to $P_{13}$.
\end{proof}

\begin{lemma}\label{lem:third_one_one_5_1}
There are exactly two mutation-equivalence classes of Fano polygons with singularity content $\left(5,\{1\times\third\}\right)$. The mutation-equivalence classes are generated by
$$
P_{21}=\sconv{(4,1),(0,1),(-3,-1)}
\qquad\text{ and }\qquad 
P_{22}=\sconv{(3,1),(0,1),(-1,0),(0,-1)}.
$$
\end{lemma}

\begin{proof}
The primitive inner normal vectors to the edges of $P_{21}$ which contribute primitive $T$-cones are $(0,-1)$ and $(-2,7)\in M$. These generate an index-two sublattice $\Tlattice{P_{21}}$ of $M$. In the case of $P_{22}$ the relevant inner normal vectors are $(1,1)$, $(1,-1)$, and $(0,-1)\in M$, and we have that $[M:\Tlattice{P_{22}}]=1$. By Lemma~\ref{lem:sublattice_index_preserved} we conclude that $P_{21}$ and $P_{22}$ lie in distinct mutation-equivalence classes.
\end{proof}

\begin{thm}\label{thm:classes_third_one_one}
There are $26$ mutation-equivalence classes of Fano polygons with singularity content $\big(n,\{m\times\third\}\big)$, $m\geq 1$. Representative polygons for each mutation-equivalence class are given in Table~\ref{tab:third_one_one_polygons}.
\end{thm}

\begin{proof}
The values $n$ and $m$ distinguishes every mutation class of Fano polygons with $\third$ singularities except in the cases $n=6$, $m=2$ and $n=5$, $m=1$. These two exceptional cases are handled in Lemmas~\ref{lem:third_one_one_6_2} and~\ref{lem:third_one_one_5_1} above. We shall show that the minimal polygons in Table~\ref{tab:minimal_third_one_one} with $n=9$, $m=1$ are connected by mutation; the remaining cases are similar. There are four minimal polytopes to consider, with numbers \hyperlink{row:4.1}{$(4.1)$}, \hyperlink{row:4.2}{$(4.2)$}, \hyperlink{row:4.3}{$(4.3)$}, and \hyperlink{row:4.4}{$(4.4)$} in Table~\ref{tab:minimal_third_one_one}. Denote these polytopes by $P_{4.1}$, $P_{4.2}$, $P_{4.3}$, and $P_{4.4}$ respectively. Then, up to $\GL_2(\Z)$-equivalence, we have the sequence of mutations:
\begin{center}
\begin{tabular}{@{}c@{\,}c@{\,}c@{\,}c@{\,}c@{\,}c@{}c@{}}
$P_{4.1}$&&&&$P_{4.2}$&\\
\raisebox{15pt}{\includegraphics[scale=0.4]{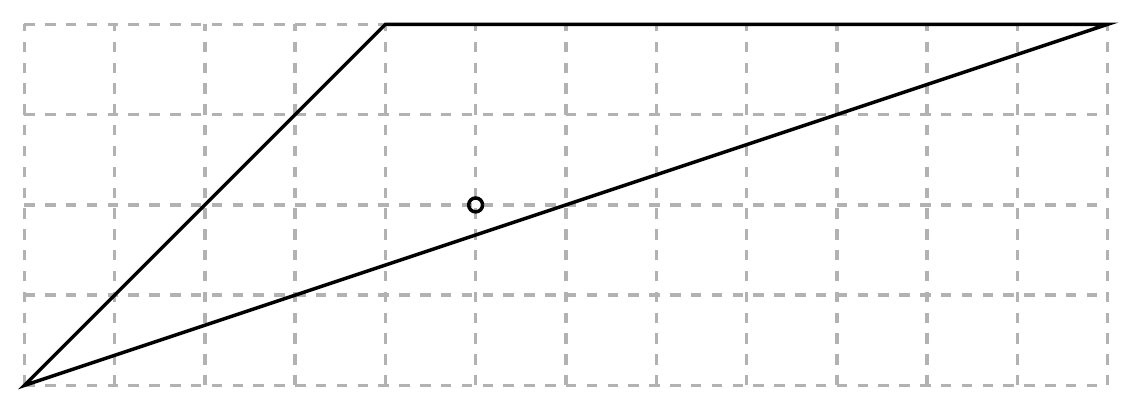}}&
\raisebox{36pt}{$\xmapsto{(0,-1)}$}&
\begin{tabular}{@{}l@{}}
  \includegraphics[scale=0.4]{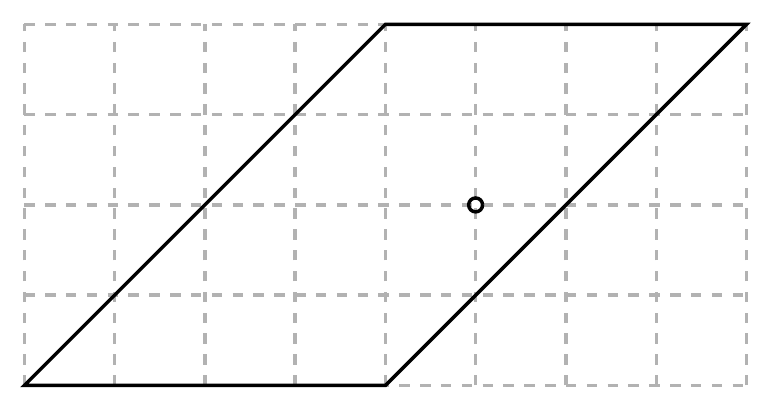}\\
  \hspace{52.5pt}\raisebox{10pt}{\rotatebox{-90}{$\longmapsto$}}$\,\scriptstyle(0,-1)$\\
  \vspace{1pt}\includegraphics[scale=0.4]{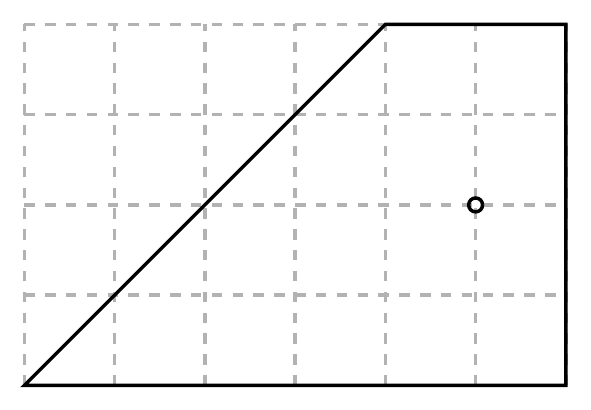}\raisebox{21pt}{$P_{4.4}$}
\end{tabular}&
\raisebox{36pt}{$\xmapsto{(1,-1)}$}&
\raisebox{4.5pt}{\includegraphics[scale=0.4]{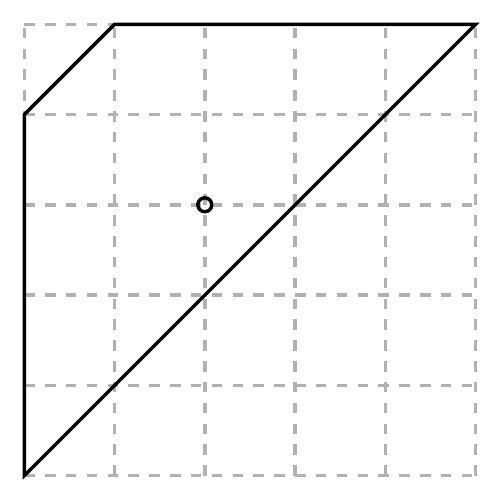}}&
\raisebox{36pt}{$\xmapsto{(1,0)}$}&
\hspace{-18pt}\raisebox{10pt}{
  \begin{tabular}{@{}l@{}}
    \phantom{$P_{4.3}$}\includegraphics[scale=0.4]{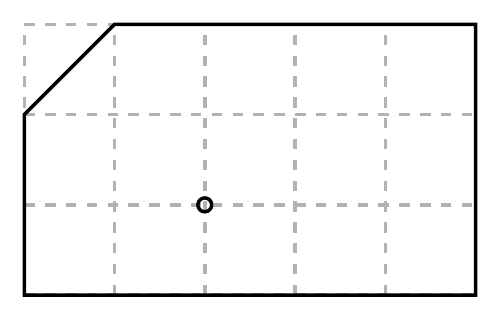}\\
    \hspace{21.5pt}\phantom{$P_{4.3}$}\raisebox{10pt}{\rotatebox{-90}{$\longmapsto$}}$\,\scriptstyle(0,-1)$\\
    \vspace{1pt}\raisebox{13pt}{$P_{4.3}$}\includegraphics[scale=0.4]{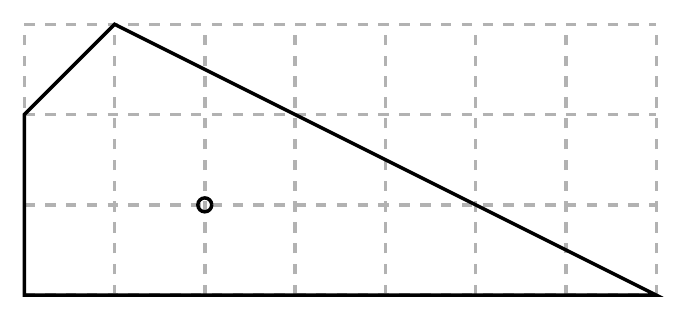}
  \end{tabular}
}
\end{tabular}
\end{center}
The mutations have been labelled with their corresponding primitive inner normal vector $w$.
\end{proof}
%-------------------------------------------------------------------------------
\begin{table}[htdp]
\centering
\caption{Representatives for the $26$ mutation-equivalence classes of Fano polygons $P\subset\NQ$ with singularity content $\left(n,\{m\times\third\}\right)$, $m\geq 1$. The degrees $(-K_X)^2=12-n-\frac{5m}{3}$ of the corresponding toric varieties are also given. See also Figure~\ref{fig:third_one_one_polygons}.}
\label{tab:third_one_one_polygons}
\setlength{\extrarowheight}{0.1em}
\begin{tabular}{rlccc}
\toprule
\multicolumn{1}{c}{\#}&\multicolumn{1}{c}{$\V{P}$}&$n$&$m$&$(-K_X)^2$\\
\cmidrule(lr){1-1} \cmidrule(lr){2-2} \cmidrule(lr){3-5}
\oddrow \padding \hyperlink{poly:1}{$1$}&\gap $(7,5),(-3,5),(-3,-5)$&$10$&$1$&$\frac{1}{3}$\\ % 1
\evnrow \padding \hyperlink{poly:2}{$2$}&\gap $(3,2),(-3,2),(-3,-2),(3,-2)$&$8$&$2$&$\frac{2}{3}$\\ % 2
\oddrow \padding \hyperlink{poly:3}{$3$}&\gap $(3,1),(3,2),(-1,2),(-2,1),(-2,-3),(-1,-3)$&$6$&$3$&$1$\\ % 3
\evnrow \padding \hyperlink{poly:4}{$4$}&\gap $(3,2),(-1,2),(-2,1),(-2,-3)$&$9$&$1$&$\frac{4}{3}$\\ % 4
\oddrow \padding \hyperlink{poly:5}{$5$}&\gap $(2,1),(1,2),(-1,2),(-2,1),(-2,-1),(-1,-2),(1,-2),(2,-1)$&$4$&$4$&$\frac{4}{3}$\\ % 5
\evnrow \padding \hyperlink{poly:6}{$6$}&\gap $(3,2),(-1,2),(-2,1),(-2,-1),(-1,-2)$&$7$&$2$&$\frac{5}{3}$\\ % 6
\oddrow \padding \hyperlink{poly:7}{$7$}&\gap $(2,1),(1,2),(-1,2),(-2,1),(-2,-1),(-1,-2),(1,-1)$&$2$&$5$&$\frac{5}{3}$\\ % 7
\evnrow \padding \hyperlink{poly:8}{$8$}&\gap $(2,1),(1,2),(-1,2),(-2,1),(-2,-1),(-1,-2)$&$5$&$3$&$2$\\ % 8
\oddrow \padding \hyperlink{poly:9}{$9$}&\gap $(1,1),(-1,2),(-2,1),(-1,-1),(1,-2),(2,-1)$&$0$&$6$&$2$\\ % 9
\evnrow \padding \hyperlink{poly:10}{$10$}&\gap $(1,1),(-1,2),(-1,-2),(1,-2)$&$8$&$1$&$\frac{7}{3}$\\ % 10
\oddrow \padding \hyperlink{poly:11}{$11$}&\gap $(1,1),(-1,2),(-2,1),(-1,-1),(2,-1)$&$3$&$4$&$\frac{7}{3}$\\ % 11
\evnrow \padding \hyperlink{poly:12}{$12$}&\gap $(3,1),(-3,1),(0,-1)$&$6$&$2$&$\frac{8}{3}$\\ % 12
\oddrow \padding \hyperlink{poly:13}{$13$}&\gap $(1, 1),(-1, 2),(-1, -1),(2, -1)$&$6$&$2$&$\frac{8}{3}$\\ % 13
\evnrow \padding \hyperlink{poly:14}{$14$}&\gap $(1,1),(-1,2),(-2,1),(-1,-1),(1,-1)$&$4$&$3$&$3$\\ % 14
\oddrow \padding \hyperlink{poly:15}{$15$}&\gap $(1,1),(-1,2),(-1,-1),(1,-1)$&$7$&$1$&$\frac{10}{3}$\\ % 15
\evnrow \padding \hyperlink{poly:16}{$16$}&\gap $(1, 1),(-1, 2),(-1, 0),(0, -1),(2, -1)$&$5$&$2$&$\frac{11}{3}$\\ % 16
\oddrow \padding \hyperlink{poly:17}{$17$}&\gap $(1,0),(1,1),(-1,2),(-2,1),(-1,-1),(0,-1)$&$3$&$3$&$4$\\ % 17
\evnrow \padding \hyperlink{poly:18}{$18$}&\gap $(1,0),(0,1),(-1,1),(-1,-3)$&$6$&$1$&$\frac{13}{3}$\\ % 18
\oddrow \padding \hyperlink{poly:19}{$19$}&\gap $(1, 1),(-1, 2),(-1, 1),(0, -1),(2, -1)$&$4$&$2$&$\frac{14}{3}$\\ % 19
\evnrow \padding \hyperlink{poly:20}{$20$}&\gap $(1,1),(-1,2),(-2,1),(-1,-1),(0,-1)$&$2$&$3$&$5$\\ % 20
\oddrow \padding \hyperlink{poly:21}{$21$}&\gap $(1,1),(-1,2),(-1,-2)$&$5$&$1$&$\frac{16}{3}$\\ % 21
\evnrow \padding \hyperlink{poly:22}{$22$}&\gap $(1, 1),(-1, 2),(-1, -1),(0, -1)$&$5$&$1$&$\frac{16}{3}$\\ % 22
\oddrow \padding \hyperlink{poly:23}{$23$}&\gap $(1, 1),(-1, 2),(0, -1),(2, -1)$&$3$&$2$&$\frac{17}{3}$\\ % 23
\evnrow \padding \hyperlink{poly:24}{$24$}&\gap $(0,1),(-1,2),(-2,1),(-1,0),(1,-1)$&$4$&$1$&$\frac{19}{3}$\\ % 24
\oddrow \padding \hyperlink{poly:25}{$25$}&\gap $(0,1),(-1,2),(-2,1),(1,-1)$&$3$&$1$&$\frac{22}{3}$\\ % 25
\evnrow \padding \hyperlink{poly:26}{$26$}&\gap $(-1,2),(-2,1),(1,-1)$&$2$&$1$&$\frac{25}{3}$\\ % 26
\bottomrule
\end{tabular}
\end{table}
%-------------------------------------------------------------------------------
\begin{figure}[htdp]
\caption{Representatives for the $26$ mutation-equivalence classes of Fano polygons $P\subset\NQ$ with singularity content $\left(n,\{m\times\third\}\right)$, $m\geq 1$. See also Table~\ref{tab:third_one_one_polygons}.}
\label{fig:third_one_one_polygons}
\centering
\vspace{0.5em}
\renewcommand{\arraystretch}{0.8}
\begin{tabular}{cc}
\hypertarget{poly:1}{}$1$&
\hypertarget{poly:2}{}$2$\\
\includegraphics[scale=0.5]{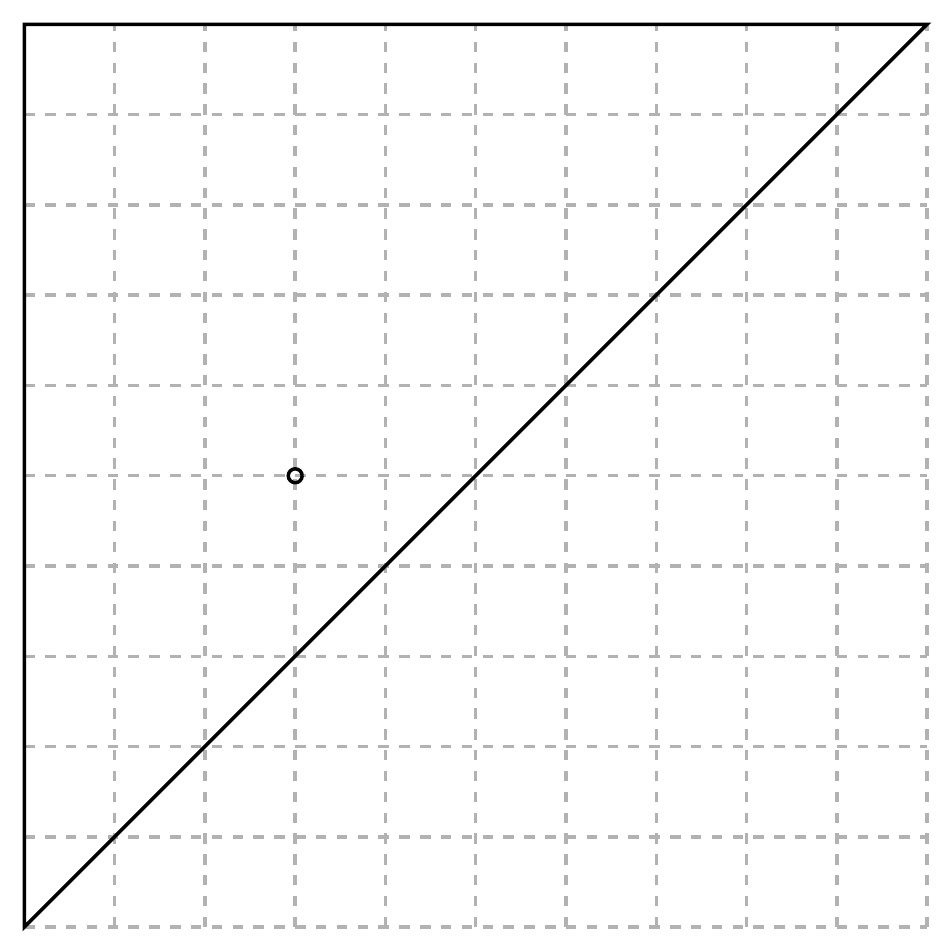}&% original height = 274pt
\raisebox{78pt}{\includegraphics[scale=0.5]{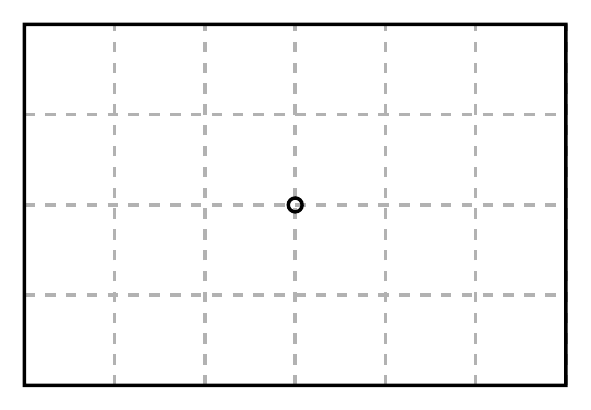}}\\% original height = 118pt
\end{tabular}
\vgap

\begin{tabular}{ccccc}
\hypertarget{poly:3}{}$3$&
\hypertarget{poly:4}{}$4$&
\hypertarget{poly:5}{}$5$&
\hypertarget{poly:6}{}$6$&
\hypertarget{poly:7}{}$7$\\
\includegraphics[scale=0.5]{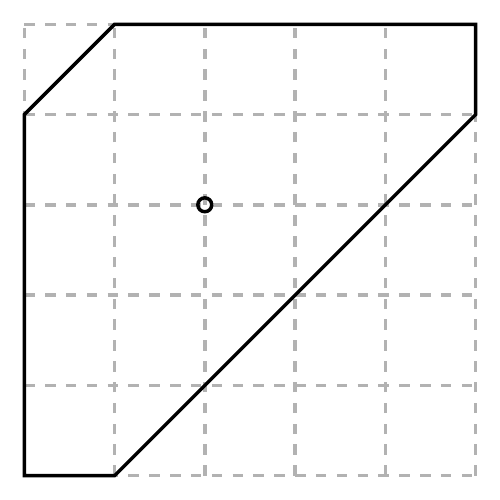}&% original height = 144pt
\includegraphics[scale=0.5]{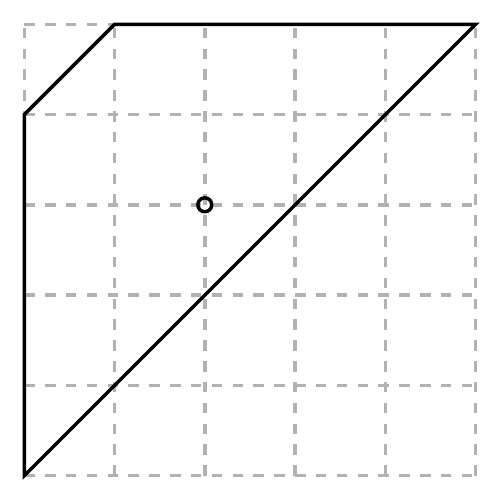}&% original height = 144pt
\raisebox{13pt}{\includegraphics[scale=0.5]{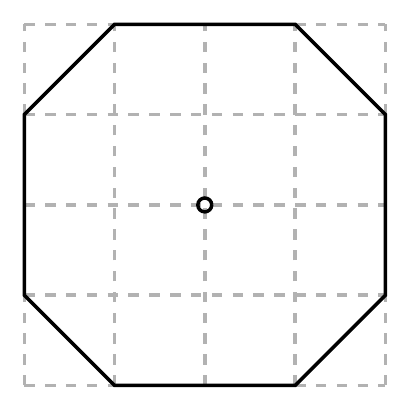}}&% original height = 118pt
\raisebox{13pt}{\includegraphics[scale=0.5]{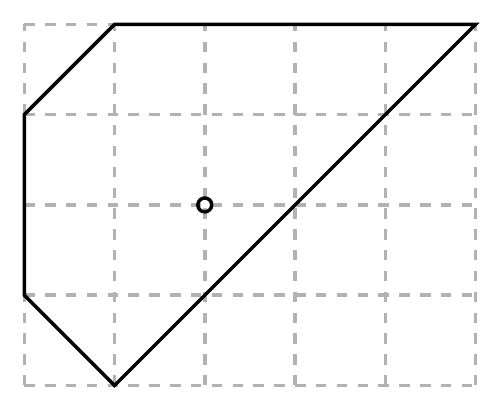}}&% original height = 118pt
\raisebox{13pt}{\includegraphics[scale=0.5]{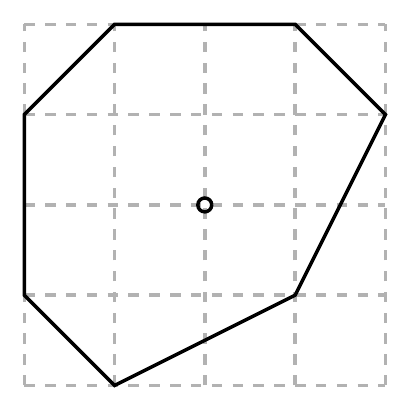}}\\% original height = 118pt
\end{tabular}
\vgap

\begin{tabular}{cccccc}
\hypertarget{poly:8}{}$8$&
\hypertarget{poly:9}{}$9$&
\hypertarget{poly:10}{}$10$&
\hypertarget{poly:11}{}$11$&
\hypertarget{poly:12}{}$12$&
\hypertarget{poly:13}{}$13$\\
\includegraphics[scale=0.5]{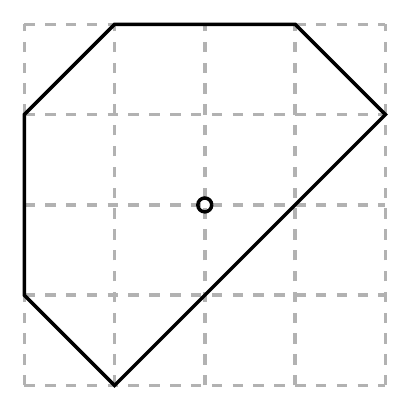}&% original height = 118pt
\includegraphics[scale=0.5]{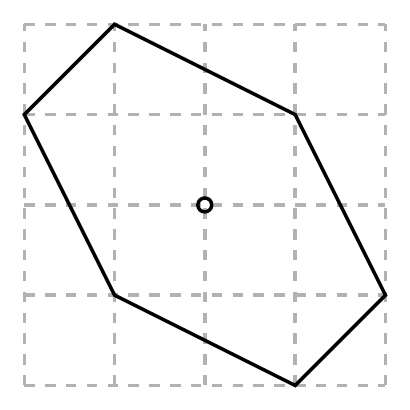}&% original height = 118pt
\includegraphics[scale=0.5]{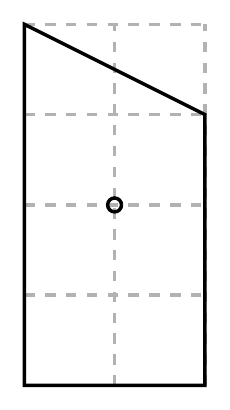}&% original height = 118pt
\raisebox{13pt}{\includegraphics[scale=0.5]{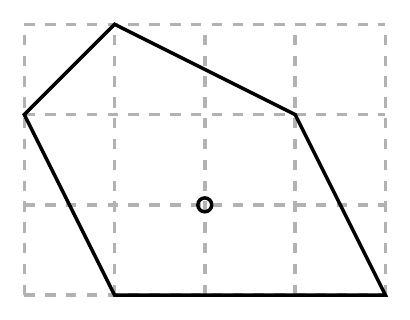}}&% original height = 92pt
\raisebox{26pt}{\includegraphics[scale=0.5]{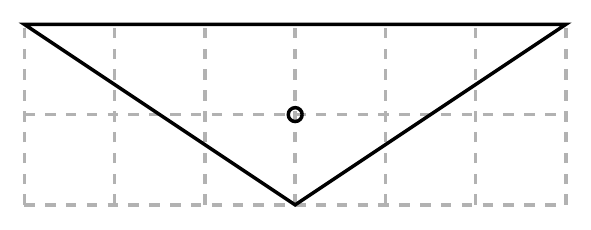}}&% original height = 66pt
\raisebox{13pt}{\includegraphics[scale=0.5]{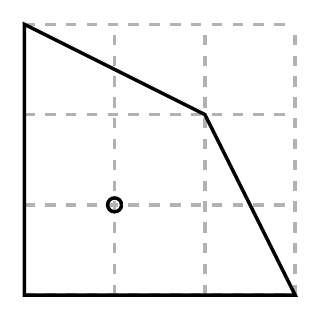}}\\% original height = 92pt
\end{tabular}
\vgap

\begin{tabular}{ccccccc}
\hypertarget{poly:14}{}$14$&
\hypertarget{poly:15}{}$15$&
\hypertarget{poly:16}{}$16$&
\hypertarget{poly:17}{}$17$&
\hypertarget{poly:18}{}$18$&
\hypertarget{poly:19}{}$19$&
\hypertarget{poly:20}{}$20$\\
\raisebox{13pt}{\includegraphics[scale=0.5]{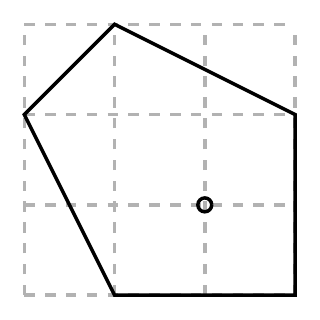}}&% original height = 92pt
\raisebox{13pt}{\includegraphics[scale=0.5]{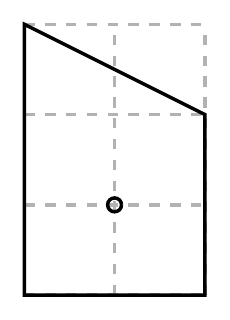}}&% original height = 92pt
\raisebox{13pt}{\includegraphics[scale=0.5]{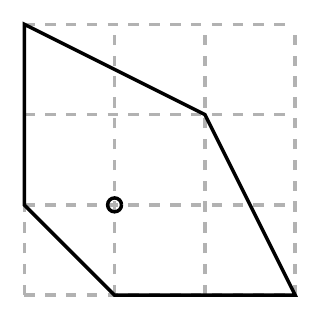}}&% original height = 92pt
\raisebox{13pt}{\includegraphics[scale=0.5]{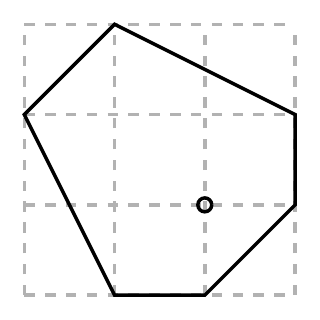}}&% original height = 92pt
\includegraphics[scale=0.5]{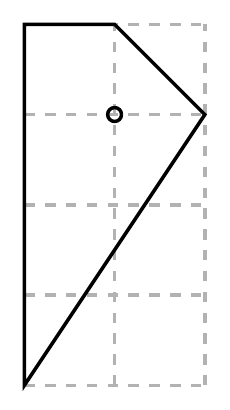}&% original height = 118pt
\raisebox{13pt}{\includegraphics[scale=0.5]{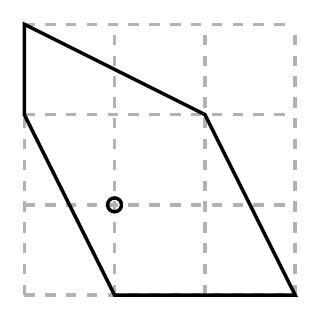}}&% original height = 92pt
\raisebox{13pt}{\includegraphics[scale=0.5]{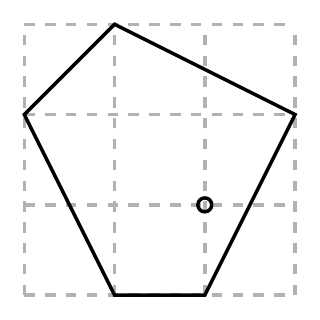}}\\% original height = 92pt
\end{tabular}
\vgap

\begin{tabular}{ccccccc}
\hypertarget{poly:21}{}$21$&
\hypertarget{poly:22}{}$22$&
\hypertarget{poly:23}{}$23$&
\hypertarget{poly:24}{}$24$&
\hypertarget{poly:25}{}$25$&
\hypertarget{poly:26}{}$26$\\
\includegraphics[scale=0.5]{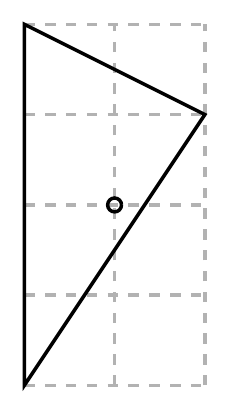}&% original height = 118pt
\raisebox{13pt}{\includegraphics[scale=0.5]{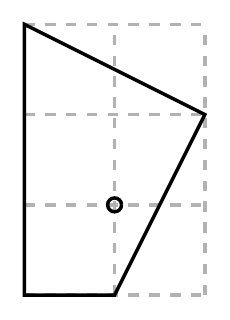}}&% original height = 92pt
\raisebox{13pt}{\includegraphics[scale=0.5]{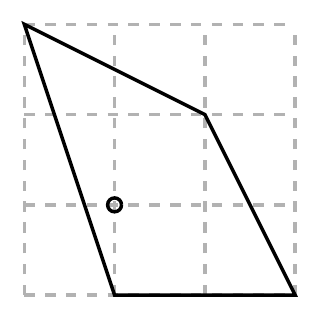}}&% original height = 92pt
\raisebox{13pt}{\includegraphics[scale=0.5]{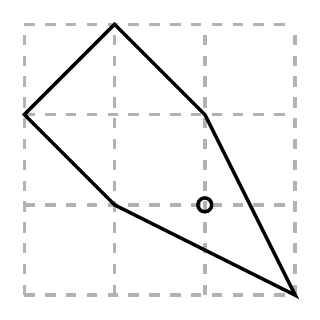}}&% original height = 92pt
\raisebox{13pt}{\includegraphics[scale=0.5]{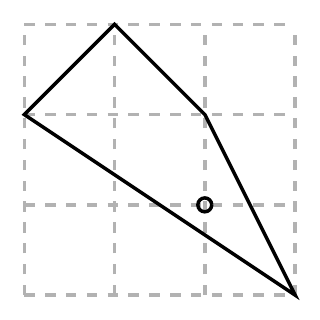}}&% original height = 92pt
\raisebox{13pt}{\includegraphics[scale=0.5]{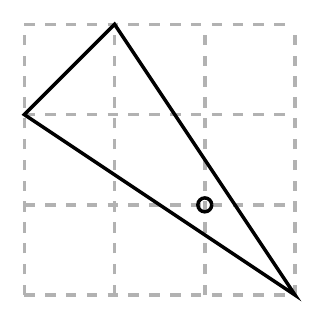}}\\% original height = 92pt
\end{tabular}
\end{figure}

%-------------------------------------------------------------------------------
\subsection*{Acknowledgements}\label{sec:acknowledgements}
%-------------------------------------------------------------------------------
This work was started at the PRAGMATIC 2013 Research School in Algebraic Geometry and Commutative Algebra, ``Topics in Higher Dimensional Algebraic Geometry'', held in Catania, Italy, during September $2013$. We are very grateful to Alfio Ragusa, Francesco Russo, and Giuseppe Zappal\'a, the organisers of the PRAGMATIC school, for creating such a wonderful atmosphere in which to work. We thank Tom Coates and Alessio Corti for many useful conversations. The majority of this paper was written during a visit by AK to Stockholm University, supported by BN's Start-up Grant. AK is supported by EPSRC grant~EP/I008128/1 and ERC Starting Investigator Grant number~240123. BN is partially supported by the Vetenskapsr{\aa}det grant~NT:2014-3991. TP is supported by an EPSRC Prize Studentship.
%-------------------------------------------------------------------------------
\bibliographystyle{plain}
\def\cprime{$'$}

%-------------------------------------------------------------------------------
\end{document}